\documentclass[pdflatex,sn-mathphys-num]{sn-jnl}

\let\vec\relax
\DeclareMathAccent{\vec}{\mathord}{letters}{"7E}

\usepackage{amsmath}
\usepackage{comment}
\usepackage{amsfonts}
\usepackage{amssymb}
\usepackage{mathtools}
\usepackage{mathrsfs,bbm}
\usepackage{xcolor}
\usepackage{lmodern}
\usepackage{comment}

\usepackage{graphicx}
\usepackage{epstopdf}
\usepackage{float}
\usepackage{tikz}			% Tikz

\usepackage{bm}
\usepackage{enumitem}
\usepackage[ruled,vlined,linesnumbered]{algorithm2e}
  
\usepackage{accents}

\date{ \phantom{b} \vspace{45mm}\phantom{e}}
\newcommand*{\dt}[1]{ \accentset{\mbox{\large\bfseries .}}{#1}}
\DeclareMathOperator{\bigtimes}{{\hbox{\large\sf X}}}

\def\R{{\mathbb R}}
\def\C{{\mathbb C}}
\def\I{\bf{I}}
\def\eps{\varepsilon}

\def\wh{\widehat}
\def\wt{\widetilde}

\def\M{\mathcal{M}}

\newcommand\bfr{{\mathbf r}}

\newcommand\bfI{{\mathbf I}}
\newcommand\bfA{{\mathbf A}}
\newcommand\bfB{{\mathbf B}}

\newcommand\bfF{{\mathbf F}}

\newcommand\bfK{{\mathbf K}}
\newcommand\bfL{{\mathbf L}}
\newcommand\bfM{{\mathbf M}}

\newcommand\bfP{{\mathbf P}}
\newcommand\bfQ{{\mathbf Q}}
\newcommand\bfR{{\mathbf R}}
\newcommand\bfS{{\mathbf S}}
\newcommand\bfU{{\mathbf U}}
\newcommand\bfV{{\mathbf V}}
\newcommand\bfW{{\mathbf W}}

\newcommand\bfY{{\mathbf Y}}
\newcommand\bfZ{{\mathbf Z}}

\newcommand\bfSigma{{\mathbf \Sigma}}

\newcommand\mat{\textbf{Mat}}
\newcommand\ten{\mathrm{Ten}}

\theoremstyle{plain}
\newtheorem{theorem}{Theorem}[section]
\newtheorem{lemma}[theorem]{Lemma}

\theoremstyle{definition}
\newtheorem{definition}[theorem]{Definition}

\theoremstyle{remark}
\newtheorem{remark}[theorem]{Remark}

\def\eps{\varepsilon}

\def\phi{\varphi}

\makeatletter
\newcommand{\ostar}{\mathbin{\mathpalette\make@circled\star}}
\newcommand{\make@circled}[2]{%
  \ooalign{$\m@th#1\smallbigcirc{#1}$\cr\hidewidth$\m@th#1#2$\hidewidth\cr}%
}
\newcommand{\smallbigcirc}[1]{%
  \vcenter{\hbox{\scalebox{0.77778}{$\m@th#1\bigcirc$}}}%
}
\makeatother

\begin{document}

\title{Robust second-order time integration of tree tensor networks}

\author[1]{\fnm{Jonas Kusch} \sur{}}\email{jonas.kusch@uni-paderborn.de}

\author[2]{\fnm{Dominik Sulz} \sur{}}\email{dominik.sulz@tum.de}

\affil[1]{\orgdiv{Numerische Mathematik}, \orgname{Universität Paderborn}, \orgaddress{\street{Warburger Str. 100}, \city{Paderborn}, \postcode{33098}, \state{Nordrhein-Westfalen}, \country{Germany}}}

\affil[2]{\orgdiv{Department of Mathematics}, \orgname{Technical University of Munich}, \orgaddress{\street{Boltzmannstraße 3}, \city{Garching bei München}, \postcode{85748}, \state{Bayern}, \country{Germany}}}

\abstract{We propose and analyze two second-order basis-update \& Galerkin (BUG) time integration methods for dynamical low-rank approximation on Tucker tensors and tree tensor networks. Both are built from a sequential pre-computation sweep, after which the differential equations for all basis matrices and connecting tensors are evolved by a Galerkin method. Further, the proposed methods are rank-adaptive by construction. The first, the second-order parallel BUG integrator, solves all these differential equations fully in parallel, which is favorable on parallel architectures, followed by a sequential augmentation and truncation step. The second, the second-order augmented BUG integrator, gives up full parallelism but conserves norm and energy for Schrödinger equations and dissipates energy for gradient flows up to the truncation tolerance. For both integrators, we prove a second-order error bound that is robust with respect to small singular values of the matricization of connecting tensors. Numerical experiments for radiative transfer and quantum spin systems validate the theoretical findings.}
    

\maketitle

%\ds{Was ich mir nicht ganz erklären kann ist, dass der 2nd order parallele BUG ähnlich gut Norm und Energie erhält wie der 2nd order augmented BUG... Hast du ne Idee warum?}
%\jknote{Beim Matrix Integrator war das auch so und man konnte zeigen, dass der Fehler mit einer sehr hohen Ordnung bzgl. $h$ gegen $0$ geht. Wahrscheinlich sehen wir einen ähnlichen Effekt hier.}

% \ds{Comments Dominik:
% \begin{itemize}
%     \item This is already the template of Numerische Mathematik
%     \item I checked Lemma 3.6 again. I think it is correct, but I do not fully get one identity (however, I guess it is trivial). The rest should be fine. 
%     \item The fact that we use $\textbf{K}_i(t_0) = \wh\bfU_i^0\wh\bfU_i^{0,*}\bfU_i^0 \bfS_i^0 \bfQ_i^{\top}\overline{\widehat\bfQ_i}$ as the initial data doesn't hurt the proofs. Even better, now they are more correct :)
%     \item So far I could not make a nice error plot with the spin systems. I am still working on this. 
%     \item The rest of the paper looks quite good I would say.
%     \item I am on holiday from tomorrow on and will be back to the office on 17th of august.
%     \item Der second-order augmented BUG sollte Norm und Energie erhaltend sein, der parallele nicht. Der Beweis für Norm und Energieerhaltung sollte der gleiche sein wie im augmented BUG paper.
%     \item Auch die Eigenschaft, garantierten Energieverlust für Gradientensysteme zu haben sollte erfüllt sein.
% \end{itemize}}

\section{Introduction}\label{sec:intro}
In this work, we derive and present numerical methods of second-order in time to approximate the solution of tensor evolution equations of the form
\begin{align}
    \dot A(t) = F(t,A(t)), \quad A(t_0) = A_0 \in \C^{n_1 \times \dots \times n_d}. \label{eq:evol_eq_intro}
\end{align}
Direct numerical treatment of this equation is infeasible for large dimensions $d$, which is often referred to as the curse of dimensionality. An approach that has proven highly effective in overcoming this issue is dynamical low-rank approximation (DLRA) \cite{KochLubich07,KochLubich10}, which approximates and evolves the solution in a low-rank Tucker \cite{tucker1966} or tree tensor network \cite{WaT03, CeLW20} format. Tree tensor networks are a data-sparse hierarchical tensor decomposition. The time evolution of their low-rank factors is determined by the Dirac-Frenkel time-dependent variational principle \cite{lubich2008quantum}. There, the right-hand side of \eqref{eq:evol_eq_intro} is projected orthogonally onto the tangent space of the manifold $\mathcal{M}$ of tree tensor networks of a fixed tree rank. Equivalently, this can be interpreted as solving the evolution equation \eqref{eq:evol_eq_intro} projected onto the manifold $\mathcal{M}$. However, a main drawback is the ill-conditioning of the resulting low-rank time evolution equations due to the presence of small singular values of matricizations of connecting tensors, corresponding to regions of high curvature in $\mathcal{M}$. 

In recent years, many novel integrators have been developed to move only on flat subspaces of the low-rank manifold, thereby removing the additional stiffness that stems from high curvature regions in the manifold and leading to time integration schemes which are robust even in the presence of small singular values. The first integrator was the projector-splitting integrator \cite{LubichOseledets,KieriLubichWalach}, with extensions to Tucker tensors in \cite{LubichVandWalach} and tree tensor networks in \cite{CeLW20}. These integrators split the dynamics into several stable substeps, simultaneously evolving the basis and coefficients of the low-rank factorization. Later, the class of basis-update \& Galerkin (BUG) integrators was derived. BUG integrators offer improved parallelism and share the benefit of evolving the solution only forward in time, which becomes a crucial property for dissipative problems and even stabilized discretizations for transport processes \cite{kusch2021stability}. While the commonly used augmented BUG integrator \cite{CeKL22,CLS23} enables parallel updates of bases, followed by sequential updates of connecting tensors, parallel BUG integrators \cite{CeKL23,ceruti2024parallel} enable the evolution of \emph{all} low-rank factors simultaneously. This inherent parallelism offers the potential to significantly accelerate the time evolution of larger tensor approximations, especially when expensive time integration schemes are required.

All integrators mentioned so far are of first order. For matrices, a robust second-order BUG based on the midpoint rule \cite{ceruti2024robust} (and another variant of higher-order extension in \cite{nobile2026high}) and a second-order parallel matrix integrator \cite{kusch2024second} have been proposed. The midpoint BUG matrix integrator \cite{ceruti2024robust}, and its generalization to matrix product states (MPS) in \cite{kapás2026}, attains second order by explicitly approximating the solution at the midpoint as a predictor, which requires three sequential substeps. The second-order parallel BUG from  \cite{kusch2024second} instead encodes the half-step information in the bases through a pre-augmentation step, and evolves all quantities at rank $2r$.
An extension to Tucker tensors and tree tensor networks of second-order, however, has been missing. This work closes this gap and derives two provably robust second-order BUG integrators for tensor-valued problems. For this, we combine the findings of first-order parallel tensor integrators from \cite{ceruti2024parallel}, the first-order augmented BUG \cite{CeKL22,CLS23}, with the derivation of a second-order parallel matrix integrator \cite{kusch2024second}. The first derived integrator, the second-order parallel BUG integrator, solves all differential equations fully in parallel while only requiring the evolution of objects at twice their original rank to achieve robust second-order accuracy and an improved dependence on normal components. The second derived integrator, the second-order augmented BUG integrator, gives up full parallelism but conserves norm and energy for Schrödinger equations and dissipates energy for gradient flows up to the truncation tolerance. Further, both integrators are rank-adaptive. The core novelty of both integrators is a carefully designed pre-augmentation that ensures the approximation of half-step information in every differential equation solved, while not requiring prohibitive memory and computational costs.

This paper is structured as follows. After the introduction in Section~\ref{sec:intro}, the second-order parallel matrix integrator is reviewed in Section~\ref{sec:recapParallel}. Section~\ref{sec:parallelTucker} presents the extension of this integrator to Tucker tensors, together with the derivation of a robust error bound. Section~\ref{sec:2nd_order_augmented_BUG} derives the second-order augmented BUG integrator for Tucker tensors together with a robust error bound. Then, Section~\ref{sec:ttn} extends both derivations to tree tensor networks. The efficiency of the presented integrators is demonstrated by numerical experiments for radiation transport and quantum spin systems in Section~\ref{sec:num_ex}.

\section{Recap: The second-order parallel basis-update \& Galerkin integrator}\label{sec:recapParallel}
We start by reviewing the second-order parallel BUG integrator of \cite{kusch2024second} for a matrix-valued problem of the form 
\begin{align*}
\dot{\bfA}(t) = \bfF(\bfA(t)), \quad \bfA(t=0) = \bfA_0\, ,
\end{align*}
where $\bfA(t)\in\mathbb{C}^{m \times n}$ and $\bfF : \mathbb{C}^{m\times n} \rightarrow \mathbb{C}^{m\times n}$. In the presentation of the integrator and throughout this manuscript, we denote tensors in italic capitals, matrices in boldface capitals and the concatenation of two matrices $\bfA\in\mathbb{C}^{m\times r_A}$ and $\bfB\in\mathbb{C}^{m\times r_B}$ as $[\bfA, \bfB] \in\mathbb{C}^{m\times (r_A + r_B)}$. Further, for a matrix $\bfM$, we denote by $\bfM^*$ the adjoint, by $\bfM^\top$ the transpose, and by $\overline{\bfM}$ the complex conjugation, i.e., $\bfM^* = \overline{\bfM}^\top$. 

Given the factored rank-$r$ numerical solution $\bfY^0 = \bfU_0 \bfS_0 \bfV_0^{*}$ at time $t_0$, one step of the second-order parallel BUG integrator from $t_0$ to $t_1 = t_0 + h$ reads:
	\begin{enumerate}
        \item Construct augmented basis matrices $\wh\bfU_0 = \text{orth}([\bfU_0, \bfF(t_0,\bfY^0)\bfV_0])$ and $\wh\bfV_0 = \text{orth}([\bfV_0, \bfF(t_0,\bfY^0)^{*}\bfU_0])$.\label{step:init_augment}
		\item 
		Evolve the rank $2r$ matrix differential equations for $\bfK$, $\bfL$, $\bar\bfS$ to determine $\widehat{\mathbf{U}}\in \mathbb{R}^{m\times 4r}$ and $\widehat{\mathbf{V}}\in \mathbb{R}^{n\times 4r}$ as well as the coefficient matrix $\bar{\mathbf{S}}(t_1) \in\mathbb{R}^{2r \times 2r}$ (in parallel):
		%Update and extend $ \bfU_0 \rightarrow \wh \bfU$ and $ \bfV_0 \rightarrow \wh\bfV$ in parallel:
		\\[2mm]
		\textbf{K-step}:
		From $t=t_0$ to $t_1$ integrate the $m \times 2r$ matrix differential equation 
		\begin{equation}\label{K-step_matrix} 
		\dot{\textbf{K}}(t) = \bfF(t, \textbf{K}(t) \wh\bfV_0^*) \wh\bfV_0, \qquad \textbf{K}(t_0) = \bfU_0 \bfS_0 \bfV_0^{*}\widehat{\bfV}_0.
		\end{equation}
		Construct $\wh \bfU = [\wh\bfU_0, \widetilde \bfU_1] \in \C^{m\times 4r}$ as an orthonormal basis of the range 
		of the $m\times 4r$ matrix $[\wh\bfU_0, \textbf{K}(t_1)]$ (e.g.~by QR decomposition), where $ \widetilde \bfU_1 \in \C^{m\times 2r}$ 
		is filled with zero columns if $[\bfU_0, \textbf{K}(t_1)]$ has rank less than $4r$.\\
		Compute the matrix $\mathbf{\widetilde S}_1^K = \widetilde\bfU_1^{*}\bfK(t_1)\in\C^{2r\times 2r}$.
		\\[2mm]
		\textbf{L-step}: 
		From $t=t_0$ to $t_1$ integrate the $n \times 2r$ matrix differential equation
		\begin{equation}\label{L-step_matrix} 
		\dot{\textbf{L}}(t) =\bfF(t, \wh\bfU_0 \textbf{L}(t)^*)^*  \wh\bfU_0, \qquad \bfL(t_0) = \bfV_0 \bfS_0^{*} \bfU_0^{*}\widehat{\bfU}_0. 
		\end{equation}
		Construct $\wh \bfV = [\wh\bfV_0, \widetilde \bfV_1] \in \C^{n\times 4r}$ as an orthonormal basis of the range 
		of the $n\times 4r$ matrix $[\wh\bfV_0, \textbf{L}(t_1)]$ (e.g.~by QR decomposition), where $ \widetilde \bfV_1 \in \C^{n\times 2r}$ 
		is filled with zero columns if $[\wh\bfV_0, \textbf{L}(t_1)]$ has rank less than $4r$.\\
		Compute the matrix $\mathbf{\widetilde S}_1^L = \bfL(t_1)^{*}\widetilde\bfV_1\in\C^{2r\times 2r}$.
		\\[2mm]
		\textbf{S-step}:
		From $t=t_0$ to $t_1$ integrate the $2r \times 2r$ matrix differential equation
		\begin{equation}\label{S-step} 
		\dot{\bar\bfS}(t) =  \wh\bfU_0^* \bfF(t, \wh\bfU_0 \bar\bfS(t) \wh\bfV_0^*) \wh\bfV_0, 
		\qquad \bar\bfS(t_0) = \widehat{\bfU}_0^{*}\bfU_0 \bfS_0 \bfV_0^{*}\widehat{\bfV}_0.
		 %\begin{pmatrix} {\bfS}_0 & 0  \\ 0 & 0 \end{pmatrix}.
		\end{equation}
		\item \textbf{Augment}:
		Construct the augmented coefficient matrix $\widehat \bfS_1 \in\C^{4r\times 4r}$ as
		\begin{align} \label{wh-S1}
            \widehat \bfS_1 = 
            \begin{pmatrix} 
                \bar \bfS(t_1) & \widetilde \bfS_1^L \\
                \widetilde \bfS_1^K & \mathbf{0}
            \end{pmatrix}.
        \end{align}
        \item \textbf{Truncate} according to a tolerance parameter $\vartheta$.
	\end{enumerate} 
	
Similar to \cite{ceruti2024robust}, this integrator is of second order, allows for rank-adaptivity while evolving all differential equations forward in time. The two main advantages of the parallel integrator are
\begin{itemize}
\item the requirement to solve only rank $2r$ evolution equations, compared to sequential $4r$ coefficient updates needed in the midpoint BUG integrator \cite{ceruti2024robust} to achieve an improved dependence on normal components of the right-hand-side,
\item the ability to evolve all differential equations in parallel, similar to \cite{CeKL23}\,.
\end{itemize}
Moreover, it can be shown that the parallel integrator shares the robust error bound of the BUG integrators~\cite{CeL20,CeKL22,CeKL23}, and the projector--splitting integrator~\cite{LubichOseledets,KieriLubichWalach}: We assume $\bfF$ to be Lipschitz continuous and bounded. Furthermore, normal components in the flux $\bfF$ at rank $r$ solutions are $\varepsilon_r$ small. That is, for a rank $r$ factorization $\bfZ = \bfU\bfSigma\bfV^{*}$ we have $\bfF(\bfZ) = \bfM(\bfZ) + \bfR(\bfZ)$, where $\bfM(\bfZ) := \bfP(\bfZ)\bfF(\bfZ) = \bfU\bfU^{*}\bfF(\bfZ)(\bfI-\bfV\bfV^{*})+\bfF(\bfZ)\bfV\bfV^{*}$, and we assume that $\Vert \bfR(\bfZ) \Vert\leq \varepsilon_r$. Then, according to~\cite[Theorem~4.1]{kusch2024second}, the global error at time $t_M = M\cdot h$ fulfills for the Frobenius norm $\Vert \cdot \Vert$
\begin{align*}
\Vert \bfY^M - \bfA(t_M)\Vert \leq \Vert \bfY^0 - \bfA(t_0)\Vert + c_1 h^2 + c_2 h\varepsilon_{r_{\mathrm{avg}}}+ c_3 \varepsilon_{2r_{\mathrm{avg}}} + c_4 M\vartheta,
\end{align*}
where constants $c_i$ are independent of small singular values in the numerical and exact solutions and $\varepsilon_{r_{\mathrm{avg}}}$ denotes the averaged sum over normal components at ranks $r_n$, where $r_n$ is the rank of the numerical solution at time $t_n$. 

% Moreover, the parallel integrator allows for an approximation of the normal component of $\bfF(t,\bfY)$ with respect to the tangent space of the rank-$r$ matrix manifold  at $(t,\bfY^0)$, given by
% \begin{equation}\label{eps}
% \| \bfR(\bfY^0) \| \approx \| \bfP^\perp(\bfY^0) [\wh \bfU\wh \bfU^* \bfF(\bfY^0) \wh\bfV\wh\bfV^*] \| = \| \widetilde \bfU_1^* \bfF(\bfY^0) \widetilde \bfV_1\|:=\eta.
% \end{equation}
% where $\bfP^\perp(\bfY)=\bfI-\bfP(\bfY)$. That is, the computation of $\eta$ gives an estimate of $\varepsilon$, which can be used to recompute a time update with a larger basis when $\eta$ lies above a certain threshold.

\section{Parallel Tucker tensor integrator}\label{sec:parallelTucker}
In the following, we present a parallel BUG integrator for problems of the form 
\begin{align*}
\dot{A}(t) = F(t,A(t)), \quad Y(t=0) = Y^0\, ,
\end{align*}
where $A(t)\in\mathbb{C}^{n_1 \times \cdots \times n_d}$. We recall that we write tensors in italic capitals and matrices in boldface capitals.

A crucial operation in the derivation of the integrators is $\mat_i(A) \in \C^{n_i \times n_{\neg i} }$, the $i$th matricization of a tensor $A\in\mathbb{C}^{n_1 \times \cdots \times n_d} $, where $n_{\neg i} = \prod_{j \neq i}^d n_j$. It is a reshape of the tensor $A$ where the $k$th row aligns all
entries of $A$ that have $k$ as the $i$th subscript. The inverse operation is called tensorization $\ten_i$ such that $A = \ten_i(\mat_i (A))$. 

We will denote matricizations of $F$, $A$, $Y$ in the $i$th mode by $\bfF_i(t, \bfZ) := \mat_i(F(t, \ten_i(\bfZ)))$, $\bfA_i = \mat_i(A)$, and $\bfY_i = \mat_i(Y)$. Moreover, we define $r_{\lnot i} := \prod_{j\neq i}r_j$. Finally, we want to recall the useful unfolding formula for Tucker tensors, cf. \cite{KoldaBader:TensorDec}
\begin{align}
    \mat_i \left( C \bigtimes_{j=1}^d \bfU_j \right) = \bfU_i \mat_i (C) \bigotimes_{j \neq i} \bfU_j^\top. \label{eq:unfolding_formula}
\end{align}
Note that \eqref{eq:unfolding_formula} also holds true for complex-valued tensors. 
\subsection{Formulation of the algorithm}\label{subsec:2ndorder_parallel_Tucker_formulation}
Given the rank $\bfr = (r_1, \cdots, r_d)$ numerical solution in Tucker format $Y^0 = C^0 \bigtimes_{i=1}^d \bfU_i^0$ at time $t_0$, one step of the parallel Tucker integrator from $t_0$ to $t_1$ reads:
	\begin{enumerate}
        \item \textbf{Pre-Augment}: Construct augmented basis matrices by defining $\textbf{Mat}_i(C^0)^{\top} = \wt \bfQ_i \wt \bfS_i^\top$ as well as $\bfV_i^{0,*} = \wt \bfQ_i^{\top} \bigotimes_{j\neq i}^d \bfU_j^{0,*} \in\C^{r_i \times n_{\lnot i}}$ and setting %\ds{We should use a different notation for the following $Q_i$, since we used it twice for different matrices. Maybe $\widetilde Q_i$}
        \begin{align*}
            \wh\bfU_i^0 = \text{orth}([\bfU_i^0, \bfF_i(t_0,\bfY_i^0)\bfV_i^{0}])\,.
        \end{align*}
        Compute the projected core $\widehat C_F^0 = F(t_0, Y^0)\bigtimes_{j=1}^d \wh \bfU_j^{0,*}$.
		\item \textbf{Parallel time integration}:  Construct augmented basis matrices $\wh \bfU_i^1 \in \C^{n_i\times 4r_i}$ and update core tensors $\widetilde C_i^1$, $\bar C$ from $t_0$ to $t_1$ (in parallel):
		\\[2mm]
		\textbf{K$_i$-step} (for $i = 1,\cdots, d$): Compute $\wh\bfV_i^{0}$ by factorizing
		\begin{align*}
            \bfQ_i \bfS_i^{0,\top} =\,& \textbf{Mat}_i(C^0 \bigtimes_{j\neq i}^d \wh \bfU_j^{0,*}  \bfU_j^0)^{\top} \in\mathbb{C}^{2r_{\lnot i}\times r_i} \\
		      \wh\bfQ_i \wh\bfS_i^\top =\,& \left[\bfQ_i, \textbf{Mat}_i(\widehat C_F^0)^{\top} (\bfU_i^{0,*}\wh\bfU_i^0)^{{\top}} \right] \in\mathbb{C}^{2r_{\lnot i}\times 2r_i} \\
              =\,& \left[\bfQ_i, \textbf{Mat}_i \left(\widehat C_F^0 \times_i \bfU_i^{0,*}\wh\bfU_i^0\right)^{{\top}} \right]
        \end{align*}
		with orthonormal $\bfQ_i$ and $\wh\bfQ_i$ by, e.g., a QR decomposition, and set $\wh\bfV_i^{0,*} = \wh\bfQ_i^{\top} \bigotimes_{j\neq i}^d \wh\bfU_j^{0,\top} \in\C^{2r_i \times n_{\lnot i}}$. Integrate the $n_i \times 2r_i$ matrix differential equations
		\begin{equation}\label{K-step} 
		\dot{\textbf{K}}_i(t) = \bfF_i(t, \textbf{K}_i(t) \wh\bfV_i^{0,*}) \wh\bfV_i^0, \qquad \textbf{K}_i(t_0) = \wh\bfU_i^0\wh\bfU_i^{0,*}\bfU_i^0 \bfS_i^0 \bfQ_i^{\top}\overline{\widehat\bfQ_i}\,.
		\end{equation}
		Construct $\wh \bfU_i^1 = [\wh \bfU_i^0, \widetilde \bfU_i] \in \C^{n_i\times 4r_i}$ as an orthonormal basis of the range 
		of the $n_i\times 4r_i$ matrix $[\wh\bfU_i^0, \textbf{K}_i(t_1)]$. %, where $ \widetilde \bfU_i^1 \in \C^{n_i\times r_i}$ 		is filled with zero columns if $(\bfU_i^0, \textbf{K}(t_1))$ has rank less than $2r_i$. \textcolor{brown}{D: Do we really fill up with zero columns? I thought we don't do this anymore.} \\
		Compute the tensor $\widetilde C_i^1 = \text{Ten}_i(\widetilde\bfU_i^{*}\bfK_i(t_1)\wh\bfQ_i^{\top}) \in\C^{2r_1\times\cdots\times 2r_d}$. %\textcolor{blue}{[(G) Rationale to be provided]}
        %\ds{To be checked: $\bfQ_i^\top \wh \bfQ_i$ or $\bfQ_i^* \wh \bfQ_i$ in the complex case?} \ds{Answer: Transpose is correct but there must be a $\overline{\wh \bfQ_i}$, see TTN section. -> Does this affect the proof?} 
		\\[2mm]
		\textbf{C-step}:
		From $t=t_0$ to $t_1$ integrate the $2r_1 \times \cdots \times 2r_d$ tensor differential equation
		\begin{equation}\label{C-step} 
		\dot{\bar C}(t) = F(t, \bar C(t) \bigtimes_{\ell=1}^d\wh\bfU_{\ell}^{0}) \bigtimes_{j=1}^d\wh\bfU_j^{0, *}, 
		\qquad \bar C(t_0) = C^0 \bigtimes_{j=1}^d \wh\bfU_j^{0,*}\bfU_j^{0}.
		\end{equation}
		\item \textbf{Augment}:
		Construct the augmented core tensor $\widehat C^1 \in\C^{4r_1\times\cdots\times 4r_d}$ such that
		\begin{align}
    		&\widehat C^1 \bigtimes_{j=1}^d \wh \bfU_j^{0,*}\wh \bfU_j^1 = \bar{C}(t_1)\,,\\
    		&\widehat C^1 \bigtimes_{j\neq i} \wh \bfU_j^{0,*}\wh \bfU_j^1 \times_i \widetilde \bfU_i^{*}\wh \bfU_i^1 = \widetilde C_i^1 \quad \text{for } i = 1,\cdots,d\, ,
		\end{align}
		and zero-valued entries elsewhere.
        \item \textbf{Truncate} according to a tolerance parameter $\vartheta$.
	\end{enumerate} 
    The algorithm then proceeds with a truncated low-rank factorization in Tucker format $Y^1 = C^1 \bigtimes_{i=1}^d \bfU_i^1$ at time $t_1$ and repeats the outlined steps until a final time $t_{\mathrm{end}}$ is reached. It must be noted that the method parallelizes the time evolution, which is commonly the most expensive step of the integrator. However, pre-augmentation, augmentation, and truncation steps must be computed sequentially.

    \begin{remark}
        The construction of $\wh \bfQ_i$ and $\bfQ_i$ can be interpreted as follows if we use the identities $ \bfU_i^{0,*}\wh\bfU_i^0 = [\bfI,0]$ and $\wh \bfU_i^{0}\wh \bfU_i^{0,*} \bfU_i^{0} = \bfU_i^0$:
    \begin{align*}
              \wh \bfQ_i =\,& \left[\bfQ_i, \textbf{Mat}_i \left(\widehat C_F^0 \times_i \bfU_i^{0,*}\wh\bfU_i^0\right)^{{\top}} \right] \bfR_i \\
              =\,& [\textbf{Mat}_i(C^0 \bigtimes_{j\neq i}^d \wh \bfU_j^{0,*}  \bfU_j^0)^{\top},\textbf{Mat}_i \left(\widehat C_F^0 \times_i \bfU_i^{0,*}\wh\bfU_i^0\right)^{{\top}}] \ \bfR_i \\
              =\,& [\textbf{Mat}_i(Y^0 \bigtimes_{j\neq i}^d \wh \bfU_j^{0,*} \times_i \bfU_i^{0,*})^\top, \textbf{Mat}_i (F(Y^0) \bigtimes_{j\neq i}^d \wh \bfU_j^{0,*} \times_i \bfU_i^{0,*})^\top ] \ \bfR_i \\
              =\,& \left( \bigotimes_{j \neq i} \wh\bfU_j^{0,*} \right) [\mat_i(Y^0)^\top , \mat_i(F(Y^0))^\top] \begin{pmatrix}
            \bfU_i^{0,*}\\\bfU_i^{0,*} \end{pmatrix} ^\top \bfR_i.
    \end{align*}
    Hence, $\wh \bfQ_i$ can be interpreted as spanning the mode $i$ matricization of $Y^0$ and $F(Y^0)$, which is tested against the augmented basis $\wh \bfU_j^1$ in all modes except $i$ and tested with the basis $\bfU_i^0$ in the $i$th mode. Further, $\bfQ_i$ can be interpreted as the $i$th matricization of $Y^0$, which is tested against the same bases. 
    \end{remark}

    \begin{remark}
        In the algorithm, one must guarantee that the first $r_i$ columns of $\wh \bfU_i^0$ equal $\bfU_i^0$. When using a QR factorization that operates column by column, this can be achieved by computing $\wh \bfU_i^{\mathrm{temp}} = [\bfU_i^{\mathrm{temp}}, \wt \bfV_i] = \text{orth}(\bfU_i^0, \bfF_i(t_0,\bfY_i^0)\bfV_i^0)$ and setting $\wh \bfU_i^0 = [\bfU_i^0, \wt \bfV_i]$. This preserves orthonormality of $\wh \bfU_i^0$, since $\wt \bfV_i$ and $\bfU_i^0$ are orthonormal by the column-by-column computation of the QR. Another alternative approach that does not require a column-by-column factorization is the following. \\ Suppose two matrices $\bfU_i^0$ and $\bfK_i(t_1)$, where $\bfU_i^0$ is orthogonal. Then perform the following steps:
        \begin{enumerate}
            \item Compute the null space projection $\bfP_i = \bfK_i(t_1) - \bfU_i^0  \bfU_i^{0,*} \bfK_i(t_1)$.
            \item Compute an svd of $\bfP_i = \bfV_i \bfS_i \bfW_i^*$. Set $\wt \bfV_i$ as the first $s$ columns of $\bfV_i$ which corresponding singular values are above a threshold (e.g. $10^{-10}$). 
            \item Set $\wh \bfU_i^0 = [\bfU_i^0, \wt \bfV_i]$
        \end{enumerate}
        Moreover, it is advantageous to re-orthogonalize $\bfU_i^0$ before step~1 to ensure that we perform an orthogonal projection.
    \end{remark}

    \subsection{Robust second-order error bound}
    The design of the above integrator is carefully guided by several observations that we wish to summarize in this section. We collect these observations as several lemmas that will later be used to derive the second-order robust error bound of the presented integrator. To simplify notation, we make use of $\bfF_i(t, \bfZ) := \textbf{Mat}_i(F(t, \ten_i(\bfZ)))$, $\bfA_i = \mat_i(A)$, and $\bfY_i = \mat_i(Y)$. Moreover, we define $\mathcal{V} = \C^{n_1 \times \cdots \times n_d}$, recall $n_{\lnot i} := \prod_{j\neq i}n_j$ and $r_{\lnot i} := \prod_{j\neq i}r_j$ and often omit dependence of $F$ on $t$. All norms are the Frobenius norm for either matrices or tensors. In particular, we often use
    \begin{align*}
        \| A \| := \left(\sum_{j_1,\cdots,j_d = 1}^{n_1,\cdots,n_d} |A_{j_1,\cdots, j_d}|^2\right)^{1/2} = \left(\sum_{j = 1}^{n_i}\sum_{k = 1}^{n_{\lnot i}} |\left((\bfA_i)_{jk}\right)|^2\right)^{1/2} =: \| \bfA_i\|\,.
    \end{align*}
    Throughout our discussion we make the following assumptions:
    \begin{enumerate}[label=A\arabic*, ref=A\arabic*]
	\item\label{ass:bounded} $F$ is bounded and Lipschitz continuous, i.e.,
    \begin{align*}
        \Vert F(Y) - F(\tilde Y) \Vert &\leq L \Vert Y - \tilde Y \Vert \quad &&\forall \ Y,\tilde Y \in \mathcal{V}, \\
        \Vert F(Y) \Vert &\leq B  \quad &&\forall \ Y \in \mathcal{V}.
    \end{align*}
	\item\label{ass:eps-normal} Given the projector onto the tangent space at $Y = C\bigtimes_{j=1}^d \bfU_j$ with Tucker ranks $r = (r_1,\cdots, r_d)$ as
    \begin{align}\label{eq:projectorP}
    P(Y)Z := Z \bigtimes_{j=1}^d \bfU_j\bfU_j^{*} + \sum_{k = 1}^d Z \bigtimes_{j\neq k} \bfU_j\bfU_j^{*} \times_k (\bfI - \bfU_k\bfU_k^{*})\,,
    \end{align}	
	the normal component of $F(t, Y_k)$ which we denote as $R(Y_k) = (I-P(Y_k))F(t, Y_k)$ fulfills $\Vert R(Y_k)\Vert \leq \varepsilon_{r}$ for $0\leq kh \leq T$. This directly implies that the flux of the $i$-th matrixization $\bfY_i = \mat_i(\bfY)$ has bounded normal flux components \cite{LubichVandWalach}. That is, $\bfF_i(\bfY_i) = \bfM_i(\bfY_i) + \bfR_i(\bfY_i)$ with $\Vert \bfR_i(\bfY_i)\Vert \leq \varepsilon_{r_i}^{(i)} \leq \varepsilon_r$  where for $\bfY_i = \bfU_i\bfS_i\bfV_i^{*}$ we have
    \begin{align}
        \bfM_i(\bfY_i) =\,& \bfU_i\bfU_i^{*}\bfF_i(\bfY_i)(\bfI-\bfV_i\bfV_i^{*}) + \bfF_i(\bfY_i)\bfV_i\bfV_i^{*}\,, \label{eq:M_i_projection} \\
        \bfR_i(\bfY_i) =\,& (\bfI-\bfU_i\bfU_i^{*})\bfF_i(\bfY_i)(\bfI-\bfV_i\bfV_i^{*})\, \notag.
    \end{align}
    In the following, we will denote $\bar\varepsilon_r := \sum_{j=1}^d \varepsilon_{r_j}^{(j)}$. 
	\item\label{ass:init} The error at the initial condition is bounded by $\Vert Y^0 - A(0)\Vert \leq \delta$.
	\end{enumerate}
    Since the integrator is rank-adaptive, $\varepsilon_r$ and $\varepsilon_{2r}$ change at each time step since $r$ is changing. Therefore, we define $\varepsilon_r^{\mathrm{max}}$ and $\varepsilon_{2r}^{\mathrm{max}}$ as the largest values of $\varepsilon_r$ and $\varepsilon_{2r}$ that appear in the dynamic. With these assumptions at hand, we can state a robust global second-order error bound for the proposed algorithm.
    \begin{theorem}\label{thm:global_error_bound_Tucker}
     Under the assumptions \ref{ass:bounded}-\ref{ass:init}, the error of the numerical approximation $Y_n$ at time $t_n = t_0 + nh$, obtained with n time steps of the second-order parallel Tucker tensor integrator with step size $h>0$ and rank-truncation tolerance $\vartheta$, is bounded by
     \begin{align}
         ||Y_n - A(t_n)|| \leq c_0 \delta +  c_1 h^2 + c_2 h \varepsilon_r^{\mathrm{max}} + c_3 \varepsilon_{2r}^{\mathrm{max}} + c_4 n\vartheta,
     \end{align}
     where all appearing constants $c_0,\dots,c_4$ are independent of small singular values of matricizations of the connection tensors.
 \end{theorem}
    
In the following, we prove a local error bound on the time interval $[t_0, t_1]$ where $A(t_0) = Y^0$. This can be extended to a general time $t_k$ by standard Lady Windermere's fan arguments \cite[Sections I.7 and
II.3]{HairerNorsettWanner:ODE_BOOK1}. 

We wish to provide a short overview of the idea of the proof:
\begin{itemize}
    \item A main ingredient that ensures second order accuracy is the choice of the pre-augmented basis $\widehat \bfU_i^0$ and the corresponding test matrix $\wh\bfV_i^{0}$, such that $\bfK_i(t_{1/2})\wh \bfV_0^*$ approximates the solution at the half-step $A(t_{1/2})$ up to $O(h^2 + h\bar\varepsilon_r)$ terms. This is proven in Lemma~\ref{le:approxK}. The property can later be used to show that the information propagated in $\bfK_i(t)$ can be recycled to approximate the core tensor at the next time step. Lemma~\ref{le:Fapprox} and Lemma~\ref{le:halfstepbasis} are needed to prove Lemma~\ref{le:approxK}. 
    \item Lemma~\ref{lemma:At_bound} proves a bound between the full solution and the full solution projected onto the basis $\wh \bfU_i^1, \ i=1,\dots,d$ of the proposed integrator at time $t_1$. 
    \item Lemma~\ref{lemma:A_UU_Y} finally proves a bound between the projected full dynamics and the solution of the second-order parallel integrator. 
    \item The proof of Theorem~\ref{thm:global_error_bound_Tucker} then easily follows from Lemma~\ref{lemma:At_bound} and Lemma~\ref{lemma:A_UU_Y}.
\end{itemize}

To provide a proof, let us first note a helpful property of the pre-augmented basis that will help us several times:
    \begin{lemma}\label{le:Fapprox}
        The pre-augmented basis spans the right-hand side at time $t_0$ up to normal components, that is,
        \begin{align}\label{eq:FapproximationTen}
            \| F(Y^0) - F(Y^0) \bigtimes_{j=1}^d \wh \bfU_j^{0} \wh \bfU_j^{0,*}  \|\leq \bar\varepsilon_r.
        \end{align}
        or, equivalently,  
        \begin{align}\label{eq:Fapproximation}
            \|\bfF_i(\bfY_i^0) - \wh\bfU_i^0\wh\bfU_i^{0,*}\bfF_i(\bfY_i^0) \bigotimes_{j \neq i} (\wh\bfU_j^0\wh\bfU_j^{0,*} )^\top \| \leq \bar\varepsilon_r\,.
        \end{align}
    \end{lemma}
    \begin{proof}
        With $\bfY_i^0 := \mat_i(Y^0) = \bfU_i^0\bfS_i^0\bfV_i^{0,*}$ we have
        \begin{align*}
        \| F(Y^0) - F(Y^0) \bigtimes_{j=1}^d \wh \bfU_j^{0} \wh \bfU_j^{0,*}  \| &= \| F(Y^0) \pm F(Y^0) \times_1 \wh \bfU_1^{0} \wh \bfU_1^{0,*} - F(Y^0) \bigtimes_{j=1}^d \wh \bfU_j^{0} \wh \bfU_j^{0,*}  \| \\
        &\leq \| F(Y^0) \times_1 (\bfI - \wh \bfU_1^{0} \wh \bfU_1^{0,*}) \| \\
        &+ \| (F(Y^0) - F(Y^0) \bigtimes_{j=2}^d \wh \bfU_j^{0} \wh \bfU_j^{0,*}) \times_1  \wh \bfU_1^{0} \wh \bfU_1^{0,*} \| \\
        &\leq \| (\bfI -  \wh \bfU_1^{0} \wh \bfU_1^{0,*})\bfF_1(\bfY_1^0) \| + \| (F(Y^0) - F(Y^0) \bigtimes_{j=2}^d \wh \bfU_j^{0} \wh \bfU_j^{0,*}) \| \\
        &\leq  \varepsilon_{r_1}^{(1)}+ \| (F(Y^0) - F(Y^0) \bigtimes_{j=2}^d \wh \bfU_j^{0} \wh \bfU_j^{0,*}) \|\,.
    \end{align*}
    Repeating this process over all $d$ modes yields \eqref{eq:FapproximationTen}. A matricization of this result directly gives \eqref{eq:Fapproximation}. 
    \end{proof}
    A crucial step to constructing a suitable pre-augmented basis is to ensure that $\bfA(t_{1/2})$ is spanned with local second order accuracy. With a forward Euler time discretization and the boundedness of normal components it can directly be shown that a suitable basis is $\wh \bfU_i^0$ for the range and $\bar \bfV_i = \text{orth}(\bfV_i^0, \bfF_i(\bfY_i^0)^{*}\bfU_i^0)$ for the co-range. Unfortunately, constructing the basis $\bar \bfV_i$ is computationally infeasible since, while its tensorization is of rank $r_i$ in mode $i$, it is generally of full rank in all remaining modes. It turns out, however, that we can find pre-augmented bases $\widehat \bfU_i^0$ and $\widehat \bfV_i^0$ that are computationally attainable while spanning the half step with sufficient accuracy.
    This approximation result is summarized in the following lemma:
    \begin{lemma}\label{le:halfstepbasis}
        Given the basis $\widehat \bfV_i^0$ and $\widehat \bfU_i^0$ as defined in the above algorithm for $i\in\{1,\cdots, d\}$, we have
        \begin{align*}
            \Vert \bfA_i(t_{1/2}) - \widehat \bfU_i^0\widehat \bfU_i^{0,*}\bfA_i(t_{1/2}) \widehat \bfV_i^0\widehat \bfV_i^{0,*} \Vert \leq C h^2 + h \bar\varepsilon_r + h\varepsilon_{r_i}^{(i)}\,,
        \end{align*}
        where $C$ only depends on the bound and Lipschitz constant of $F$.
    \end{lemma}
    \begin{proof}
    Using a forward Euler (FE) time discretization and boundedness of the normal component in mode $i$, we have
    \begin{align}\label{eq:A_ihalfexpand}
        \bfA_i(t_{1/2}) \stackrel{\mathrm{FE}}{=}\,& \bfY_i^0 + \frac12 h \bfF_i(\bfY_i^0) + O(h^2) \nonumber\\
        \stackrel{\ref{ass:eps-normal}}{=}\,& \bfY_i^0 + \frac{h}{2} (\bfU_i^0\bfU_i^{0,*}\bfF_i(\bfY_i^0)(\bfI - \bfV_i^0\bfV_i^{0,*}) + \bfF_i(\bfY_i^0)\bfV_i^0\bfV_i^{0,*}) + O\big(h^2 + h\varepsilon_{r_i}^{(i)}\big)\,.
    \end{align}
    To construct a computationally attainable basis to span the co-range of $\bfA_i(t_{1/2})$, we use Lemma~\ref{le:Fapprox}, in particular \eqref{eq:Fapproximation}, which yields 
    \begin{align*}
        \bfU_i^{0,*}\bfF_i(\bfY_i^0) \stackrel{\eqref{eq:Fapproximation}}{=}\,& \bfU_i^{0,*}\wh\bfU_i^0\wh\bfU_i^{0,*}\bfF_i(\bfY_i^0) \bigotimes_{j \neq i} (\wh\bfU_j^0\wh\bfU_j^{0,*} )^\top + \mathcal{O}\big(\bar\varepsilon_r\big)\\
        =\,& \bfU_i^{0,*}\wh\bfU_i^0 \,\mat_i(F(Y^0) \bigtimes_{\ell=1}^d \wh\bfU_{\ell}^{0,*}) \bigotimes_{j \neq i}\wh\bfU_j^{0,\top} + \mathcal{O}\big(\bar\varepsilon_r\big)\\
        =\,& \bfU_i^{0,*}\wh\bfU_i^0 \,\mat_i(\widehat C_F^0) \bigotimes_{j \neq i}\wh\bfU_j^{0,\top} + \mathcal{O}\big(\bar\varepsilon_r\big)\,.
    \end{align*}
    Thus, if we define the QR factorization $\bfQ\bfR = \mat_i(\widehat C_F^0)^{\top} (\bfU_i^{0,*}\wh\bfU_i^0)^{\top} \in\mathbb{C}^{2r_{\lnot i} \times r_i}$, the matrix $\bigotimes_{j \neq i} \wh\bfU_j^0 \; \bfQ$ spans $\bfF_i(\bfY_i^0)^{*}\bfU_i^0$ with sufficient accuracy while its tensorization is of rank $r_i$ in mode $i$ as well as $2r_j$ in each mode $j \neq i$. We can then conclude that
    $\wh\bfV_i^0 \in\mathbb{C}^{n_{\lnot i} \times 2r_i}$ spans the column space of $\bfA_i(t_{1/2})^{*}\in\mathbb{C}^{n_{\lnot i} \times n_i}$ due to its construction, while $\wh\bfU_i^0$ spans the column space of $\bfA_i(t_{1/2})$ with sufficient accuracy. That is, multiplying \eqref{eq:A_ihalfexpand} from the left with $\wh\bfU_i^0$ and from the right with $\wh\bfV_i^0$ spans all terms exactly, where in particular $\bfU_i^{0,*}\bfF_i(\bfY_i^0)\wh\bfV_i^0\wh\bfV_i^{0,*} = \bfU_i^{0,*}\bfF_i(\bfY_i^0)$.
    \end{proof}
    Further, let us also remark that by construction, $\widehat \bfU_i^0$ and $\widehat \bfV_i^0$ span $\bfY_i^0$. A consequence of using a sufficiently accurate basis $\widehat \bfV_i^0$ in the $\bfK$-step of mode $i$ is that $\bfK_{i}(t_{1/2})\widehat \bfV_i^{0,*}$ is a sufficiently accurate approximation of $\bfA_i(t_{1/2})$, namely
    \begin{lemma}\label{le:approxK}
    Given the assumptions above, the full half-step fulfills
    \begin{subequations}
        \begin{align}
            \Vert \bfA_i(t_{1/2}) - \bfK_i(t_{1/2})\wh\bfV_i^{0,*} \Vert \leq\,& C_1 h^2 + h\bar\varepsilon_r + h\varepsilon^{(i)}_r\,, \label{eq:approxKa}\\
            \Vert \bfA_i(t_{1/2}) - \wh\bfU_i^0\wh\bfU_i^{0,*}\bfK_i(t_{1/2})\wh\bfV_i^{0,*} \Vert \leq\,& C_2 h^2 + h\bar\varepsilon_r + h\varepsilon^{(i)}_r\,,\label{eq:approxKb}
        \end{align}
        \end{subequations}
        where $C_{1,2}$ only depend on the bound and Lipschitz constant of $F$.
    \end{lemma}
    \begin{proof}
        Applying Lemma~\ref{le:halfstepbasis}, we directly obtain
        \begin{align*}
            \Vert \bfA_i(t_{1/2}) - \bfK_i(t_{1/2})\wh\bfV_i^{0,*} \Vert \leq\,& \| \bfA_i(t_{1/2}) - \wh\bfU_i^0\wh\bfU_i^{0,*}\bfA_i(t_{1/2})\wh\bfV_i^{0}\wh\bfV_i^{0,*} \|\\
            \,&+ \| \wh\bfU_i^0\wh\bfU_i^{0,*}\bfA_i(t_{1/2})\wh\bfV_i^{0}\wh\bfV_i^{0,*} - \bfK_i(t_{1/2})\wh\bfV_i^{0,*} \| \\
            \leq\,& \| \wh\bfU_i^0\wh\bfU_i^{0,*}\bfA_i(t_{1/2})\wh\bfV_i^{0} - \bfK_i(t_{1/2}) \| + C h^2 + h \bar\varepsilon_r + h\varepsilon_{r_i}^{(i)}.
        \end{align*}
        %Using Assumption~\ref{ass:bounded}, the forward Euler method then gives
        %\begin{align*}
        %    \wh\bfU_i^0\wh\bfU_i^{0,*}\bfA_i(t_{1/2})\wh\bfV_i^{0} =\,& \wh\bfU_i^0\wh\bfU_i^{0,*}\bfY_i^0\wh\bfV_i^{0}  + \frac{h}{2} \wh\bfU_i^0\wh\bfU_i^{0,*}\bfF_i(\bfY_i^0)\wh\bfV_i^{0} + O(h^2) \\
        %    =\,& \bfY_i^0\wh\bfV_i^{0} + \frac{h}{2} \bfF_i(\bfY_i^0)\wh\bfV_i^{0}+ O(h^2)\,, \\
        %    \bfK_i(t_{1/2}) =\,& \bfK_i(t_0) + \frac{h}{2} \bfF_i(\bfK_i(t_0)\widehat \bfV_i^{0,*}) \widehat \bfV_i^0 + O(h^2)\,,
        %\end{align*}
        %where we have used that $\wh\bfU_i^0$ spans $\bfY_i^0$ and $\bfF_i(\bfY_i^0)\wh\bfV_i^{0}$ by construction. I.e. we used $\wh\bfU_i^0\wh\bfU_i^{0,*} \bfY_i^0 = \bfY_i^0$ and $\wh\bfU_i^0\wh\bfU_i^{0,*} \bfF_i(\bfY_i^0)\wh\bfV_i^{0} = \bfF_i(\bfY_i^0)\wh\bfV_i^{0}$. 
        %\ds{Gilt das nicht nur bis auf $\bar \epsilon_r$ für die Identität $\wh\bfU_i^0\wh\bfU_i^{0,*} \bfF_i(\bfY_i^0)\wh\bfV_i^{0} = \bfF_i(\bfY_i^0)\wh\bfV_i^{0}$?} \jknote{Ja richtig, daher habe ich darunter den roten Text. Das würde den Absatz vorher ersetzen. Genauso wie auch weiter unten.}
        %\ds{Meine Frage bezog sich auf die folgende Gleichung. Ist diese nicht nur richtig bis auf $\bar\eps_r$? $\wh\bfU_i^0\wh\bfU_i^{0,*} \bfF_i(\bfY_i^0)\wh\bfV_i^{0} = \bfF_i(\bfY_i^0)\wh\bfV_i^{0} + \bar\epsilon_r$? Oder warum gilt das?}\jknote{Ja richtig, das gilt nur bis auf $\varepsilon_r$. Der Text oben sollte aber eh rausgenommen und gegen den roten Text unten ersetzt werden. Der schwarze Text war falsch, der rote ist die Korrektur. Ich habe den schwarzen Text mal auskommentiert.}
        
        Using Assumption~\ref{ass:bounded}, the forward Euler method then gives
        \begin{align*}
            \wh\bfU_i^0\wh\bfU_i^{0,*}\bfA_i(t_{1/2})\wh\bfV_i^{0} =\,& \wh\bfU_i^0\wh\bfU_i^{0,*}\bfY_i^0\wh\bfV_i^{0}  + \frac{h}{2} \wh\bfU_i^0\wh\bfU_i^{0,*}\bfF_i(\bfY_i^0)\wh\bfV_i^{0} + O(h^2) \\
            =\,& \bfY_i^0\wh\bfV_i^{0} + \frac{h}{2} \wh\bfU_i^0\wh\bfU_i^{0,*}\bfP(\bfY_i^0)\bfF_i(\bfY_i^0)\wh\bfV_i^{0}+ O(h^2 + h\varepsilon_{r_i}^{(i)})\,, \\
            =\,& \bfY_i^0\wh\bfV_i^{0} + \frac{h}{2} \bfP(\bfY_i^0)\bfF_i(\bfY_i^0)\wh\bfV_i^{0}+ O(h^2 + h\varepsilon_{r_i}^{(i)})\,, \\
            =\,& \bfY_i^0\wh\bfV_i^{0} + \frac{h}{2} \bfF_i(\bfY_i^0)\wh\bfV_i^{0}+ O(h^2 + h\varepsilon_{r_i}^{(i)})\,,\\
            \bfK_i(t_{1/2}) =\,& \bfK_i(t_0) + \frac{h}{2} \bfF_i(\bfK_i(t_0)\widehat \bfV_i^{0,*}) \widehat \bfV_i^0 + O(h^2)\,,
        \end{align*}
        where we have used that $\wh\bfU_i^0$ spans $\bfY_i^0$ and $\bfP(\bfY_i^0)\bfF_i(\bfY_i^0)$ by construction. I.e., we used $\wh\bfU_i^0\wh\bfU_i^{0,*} \bfY_i^0 = \bfY_i^0$ and $\wh\bfU_i^0\wh\bfU_i^{0,*} \bfP(\bfY_i^0)\bfF_i(\bfY_i^0) = \bfP(\bfY_i^0)\bfF_i(\bfY_i^0)$ by definition of $\bfP(\bfY_i^0) \bfF_i(\bfY_i^0)$ and $\wh \bfU_i^0$, cf. \eqref{eq:M_i_projection}.
        Since $\bfK_i(t_0) = \bfY_i^0 \widehat \bfV_i^0$ and $\bfK_i(t_0)\widehat \bfV_i^{0,*} = \bfY_i^0$, we get
        \begin{align*}
        \| \wh\bfU_i^0\wh\bfU_i^{0,*}\bfA_i(t_{1/2})\wh\bfV_i^{0} - \bfK_i(t_{1/2}) \| \leq c_1 h^2 +c_2h\varepsilon_{r_i}^{(i)}
        \end{align*}
        and hence the desired bound for \eqref{eq:approxKa}. Moreover,
        \begin{align*}
            \Vert \bfA_i(t_{1/2}) - \wh\bfU_i^0\wh\bfU_i^{0,*}\bfK_i(t_{1/2})\wh\bfV_i^{0,*} \Vert \leq& \ \Vert \bfA_i(t_{1/2}) - \bfK_i(t_{1/2})\wh\bfV_i^{0,*} \Vert \\
            &+ \Vert (\bfI- \wh\bfU_i^0\wh\bfU_i^{0,*})\bfK_i(t_{1/2})\Vert\,.
        \end{align*}
        %Thus, it remains to bound the second norm. With a forward Euler step of $\bfK_i(t_{1/2})$ we obtain 
        %\begin{align}\label{eq:le2_approxK}
        %    \Vert (\bfI- \wh\bfU_i^0\wh\bfU_i^{0,*})\bfK_i(t_{1/2})\Vert \leq \Vert (\bfI- \wh\bfU_i^0\wh\bfU_i^{0,*}) (\bfY_i^0\wh\bfV_i^{0} + h\bfF_i(\bfY_i^0)\wh\bfV_i^{0} )\Vert + c_1h^2 = c_1h^2\,,
        %\end{align}
        %where we used that the norm term on the right-hand-side is zero by construction of $\wh\bfU_i^0$, which concludes the proof.
        
        Thus, it remains to bound the second norm. With a forward Euler step of $\bfK_i(t_{1/2})$ we obtain 
        \begin{align}\label{eq:le2_approxK}
            \Vert (\bfI- \wh\bfU_i^0\wh\bfU_i^{0,*})\bfK_i(t_{1/2})\Vert \leq\,& \Vert (\bfI- \wh\bfU_i^0\wh\bfU_i^{0,*}) (\bfY_i^0\wh\bfV_i^{0} + h\bfF_i(\bfY_i^0)\wh\bfV_i^{0} )\Vert + c_1h^2\nonumber \\
            =\,& \Vert (\bfI- \wh\bfU_i^0\wh\bfU_i^{0,*}) (\bfY_i^0 + h\bfP(\bfY_i^0)\bfF_i(\bfY_i^0) )\wh\bfV_i^{0}\Vert + c_1h^2 + h\varepsilon_{r_i}^{(i)}\nonumber\\
            =\,& c_1h^2 + h\varepsilon_{r_i}^{(i)}\,
        \end{align}
        where we have again used that $\wh\bfU_i^0$ spans $\bfY_i^0$ and $\bfP(\bfY_i^0)\bfF_i(\bfY_i^0)$ by construction, concluding the proof.
    \end{proof}
    Further, choosing a pre-augmented basis in the integrator that spans the solution at times $t_0$ and $t_{1/2}$ locally with second-order accuracy will directly lead to basis matrices $\wh \bfU_i^1$ that span $A(t_1)$ locally with second-order accuracy. We make this statement rigorous in the following Lemma:
\begin{lemma}\label{lemma:At_bound}
		Given the assumptions above, the full solution projected onto the basis of the parallel integrator fulfills		
        \begin{align*}
		\Vert A(t_1) - A(t_1) \bigtimes_{i=1}^d \wh \bfU_i^1 \wh \bfU_i^{1,*} \Vert \leq  c_1 h^3 + c_2 \bar\varepsilon_{2r}h + c_3 \bar\varepsilon_r h^2,
		\end{align*}
		where all arising constants only depend on the bound and Lipschitz constant of $F$, the bound of third derivatives of the exact solution, and an upper bound of the time stepsize.
	\end{lemma}
	\begin{proof}
	   First, we start by noting 
    \begin{align*}
        \| A(t_1) - A(t_1) \bigtimes_{i=1}^d \wh \bfU_i^1 \wh \bfU_i^{1,*} \| =\,& \| A(t_1) \pm A(t_1) \times_1 \wh \bfU_1^1 \wh \bfU_1^{1,*} - A(t_1) \bigtimes_{i=1}^d \wh \bfU_i^1 \wh \bfU_i^{1,*} \| \\
        \leq\,& \| A(t_1) - A(t_1) \times_1 \wh \bfU_1^1 \wh \bfU_1^{1,*} \| \\
        \,&+ \| A(t_1) \times_1 \wh \bfU_1^1 \wh \bfU_1^{1,*} - A(t_1) \bigtimes_{i=1}^d \wh \bfU_i^1 \wh \bfU_i^{1,*} \| \\
        =\,& \| A(t_1) \times_1 (\bfI - \wh \bfU_1^1 \wh \bfU_1^{1,*}) \| + \| A(t_1) - A(t_1) \bigtimes_{i=2}^d \wh \bfU_i^1 \wh \bfU_i^{1,*} \|\,.
    \end{align*}
    Therefore, by a repeated application of the above bound, we get
    \begin{align*}
        \| A(t_1) - A(t_1) \bigtimes_{i=1}^d \wh \bfU_i^1 \wh \bfU_i^{1,*} \| \leq \sum_{k=1}^d \| A(t_1) \times_k (\bfI - \wh \bfU_k^1 \wh \bfU_k^{1,*} ) \|\,.
    \end{align*}
    Thus, defining $\bfP_{i}^\perp := \bfI - \wh \bfU_i^1 \wh \bfU_i^{1,*}$, it suffices to bound terms of the form $\| A(t_1) \times_i \bfP_{i}^\perp \| = \Vert \bfP_{i}^\perp \bfA_i(t_1) \Vert$. Now, we note that by Lemma~\ref{le:approxK} and the midpoint (MP) rule we have
    \begin{align*}
        \bfA_i(t_1) \stackrel{\mathrm{MP}}{=}\,& \bfY_i^0 + h \bfF_i(\bfA_i(t_{1/2})) + O(h^3)\\
        \stackrel{\mathrm{Le.~\ref{le:approxK}}}{=}\,& \bfY_i^0 + h \bfF_i( \wh\bfU_i^0\wh\bfU_i^{0,*}\bfK_i(t_{1/2})\wh\bfV_i^{0,*} ) + O(h^3 + h^2 \bar\varepsilon_r)\,.
    \end{align*}
    From boundedness of normal components in mode $i$ we have
    \begin{align*}
        \bfA_i(t_1) \stackrel{\ref{ass:eps-normal}}{=}\,& \bfY_i^0 + h \wh\bfU_i^0\wh\bfU_i^{0,*}\bfF_i( \wh\bfU_i^0\wh\bfU_i^{0,*}\bfK_i(t_{1/2})\wh\bfV_i^{0,*} )(\bfI - \wh\bfV_i^{0}\wh\bfV_i^{0,*})\\
        \,&+ h \bfF_i( \wh\bfU_i^0\wh\bfU_i^{0,*}\bfK_i(t_{1/2})\wh\bfV_i^{0,*} )\wh\bfV_i^{0}\wh\bfV_i^{0,*} + O(h^3 + h^2 \bar\varepsilon_r + h\varepsilon_{2r_i}^{(i)})\,.
    \end{align*}    
    Following Lemma~\ref{le:approxK} (in particular inequality \eqref{eq:le2_approxK}), we have $\wh\bfU_i^0\wh\bfU_i^{0,*}\bfK_i(t_{1/2}) = \bfK_i(t_{1/2}) + O(h^2)$. Thus, with $\bfP_{i}^\perp \wh\bfU_i^0 = 0$ and $\mu :=h^3 + h^2 \bar\varepsilon_r + h\varepsilon_{2r_i}^{(i)}$ we have
    \begin{align*}
        \bfP_{i}^\perp \bfA_i(t_1) = h \bfP_{i}^\perp \bfF_i(\bfK_i(t_{1/2})\wh\bfV_i^{0,*} )\wh\bfV_i^{0}\wh\bfV_i^{0,*} + O(\mu) \,.
    \end{align*}
    Moreover, since $\bfP_{i}^\perp \bfK_i(t_1) = 0$ by construction of the augmented basis and $\bfK_i(t_1) = \bfY_i^0\wh\bfV_i^{0} + h \bfF_i(\bfK_i(t_{1/2})\wh\bfV_i^{0,*} )\wh\bfV_i^{0} + O(h^3)$ by the midpoint rule we directly have
    \begin{align*}
        \bfP_{i}^\perp \bfA_i(t_1) =\,& h \bfP_{i}^\perp \bfF_i(\bfK_i(t_{1/2})\wh\bfV_i^{0,*} )\wh\bfV_i^{0}\wh\bfV_i^{0,*} - \bfP_{i}^\perp \bfK_i(t_1)\wh\bfV_i^{0,*} + O(\mu)= O(\mu)\,,
    \end{align*}
    which concludes the proof.
    \end{proof}
Thus, we have shown that the bases generated in the $K_i$-steps suffice to approximate the solution at the next time step. It remains to discuss how the approximation of the core tensor $\widehat C^1$ approximates the optimal coefficients $\widehat C_A := A(t_1) \bigtimes_{i=1}^d \wh\bfU_i^{1,*}$. We can distinguish these coefficients in several blocks: 
\begin{align*}
    \widehat C_A \bigtimes_{i=1}^d \wh\bfU_i^{0,*}\wh\bfU_i^1 + \sum_{j=1}^d \widehat C_A \bigtimes_{i\neq j} \wh\bfU_i^{0,*}\wh\bfU_i^1 \times_j \widetilde\bfU_j^* \wh\bfU_j^1 + \Delta\,,
\end{align*}
where $\Delta$ is the sum over all terms of the form
\begin{align*}
    \widehat C_A \bigtimes_{i\in\mathcal{I}} \wh\bfU_i^{0,*}\wh\bfU_i^1 \bigtimes_{j\in\mathcal{J}} \widetilde\bfU_j^* \wh\bfU_j^1
\end{align*}
with $\mathcal{J}\subset \{1,\cdots,d\}$ such that $|\mathcal{J}| \geq 2$ and $\mathcal{I} := \{1,\cdots,d\}\backslash \mathcal{J}$. We can now show that this core tensor is approximated sufficiently accurately. 
 \begin{lemma}\label{lemma:A_UU_Y}
     Given our previous assumptions, the distance between the projected full
dynamics and the solution of the parallel integrator is bounded by
    \begin{align*}
         \| A(t_1) \bigtimes_{j=1}^d \wh \bfU_j^1 \wh \bfU_j^{1,*} - Y_1 \| \leq C ( h^3 +  (\varepsilon_{r}+\bar\varepsilon_r) h^2 + (\bar\varepsilon_{2r}+\varepsilon_{2r})h ) ,
     \end{align*}
     where constants only depend on the Lipschitz constant and bound of $F$, a bound of third derivatives of the exact solution, and an upper bound of the time stepsize.
 \end{lemma}
 \begin{proof}
    We prove this by going over all blocks of the core tensors. We start by expanding the block corresponding to $\wh \bfU_j^0$. 
    \begin{align*}
        \| A(t_1) \bigtimes_{j=1}^d \wh \bfU_j^{0,*} - \bar C(t_1) \| &\leq \int_{t_0}^{t_1} \| \left( F(t,A(t))  - F(t,\bar C(t) \bigtimes_{j=1}^d \wh \bfU_j^{0}) \right) \bigtimes_{j=1}^d \wh \bfU_j^{0,*} \| \,dt \\
        &\stackrel{\ref{ass:bounded}}{\leq} L \int_{t_0}^{t_1} \| A(t) - \bar C(t) \bigtimes_{j=1}^d \wh \bfU_j^{0} \| \,dt \\
        &\leq L \int_{t_0}^{t_1} \int_{t_0}^{t} \| F(s,A(s)) - F(s,\bar C(s) \bigtimes_{j=1}^d \wh \bfU_j^{0}) \bigtimes_{j=1}^d \wh \bfU_j^{0} \wh \bfU_j^{0,*} \| \,ds \ dt \\
        &\leq c_1 h^3 + Lh^2 \| F(Y^0) - F(Y^0) \bigtimes_{j=1}^d \wh \bfU_j^{0} \wh \bfU_j^{0,*}  \| \stackrel{\mathrm{Le.}~\ref{le:Fapprox}}{\leq} c_1 h^3 + Lh^2\bar\varepsilon_r\,,
    \end{align*}
    where $c_1$ results from a Taylor expansion of $A(s)$ and $\bar C(s)$ around $t_0$ and \ref{ass:bounded}. Next, we bound the terms
    \begin{align*}
        \Vert A(t_1) \bigtimes_{j\neq i} \wh\bfU_j^{0,*} \times_i \widetilde\bfU_i^* - \wt C_i^1\Vert =\,& \Vert \widetilde\bfU_i^*\bfA_i(t_1) \bigotimes_{j\neq i} \wh\bfU_j^{0}  - \widetilde\bfU_i^{*}\bfK_i(t_1)\wh\bfQ_i^{\top}\Vert\\
        \leq\,& h\Vert \widetilde\bfU_i^*\bfF_i(\bfA_i(t_{1/2})) \bigotimes_{j\neq i} \wh\bfU_j^{0}  - \widetilde\bfU_i^*\bfF_i(\bfK_i(t_{1/2})\wh\bfV_i^{0,*})\wh\bfV_i^{0}\wh\bfQ_i^{\top}\Vert + ch^3\,,
    \end{align*}
    where we have used the midpoint rule in the last inequality. Recalling that according to Lemma~\ref{le:approxK} we have $\Vert \wh\bfU_i^0\wh\bfU_i^{0,*}\bfK_i(t_{1/2})\wh\bfV_i^{0,*} - \bfA_i(t_{1/2}) \Vert \leq C_1 h^2 + 2 h\bar \varepsilon_r$, and thus
    \begin{align*}
        \widetilde\bfU_i^*\bfF_i(\bfA_i(t_{1/2})) \bigotimes_{j\neq i} \wh\bfU_j^{0} =\,& \widetilde\bfU_i^*\bfM_i(\wh\bfU_i^0\wh\bfU_i^{0,*}\bfK_i(t_{1/2})\wh\bfV_i^{0,*}) \bigotimes_{j\neq i} \wh\bfU_j^{0} + O(h^2 + h\bar \varepsilon_r)\\
        =\,& \widetilde\bfU_i^*\bfF_i(\wh\bfU_i^0\wh\bfU_i^{0,*}\bfK_i(t_{1/2})\wh\bfV_i^{0,*})\wh\bfV_i^{0}\wh\bfV_i^{0,*} \bigotimes_{j\neq i} \wh\bfU_j^{0} + O(h^2 + h\bar \varepsilon_r)\\
        =\,& \widetilde\bfU_i^*\bfF_i(\bfK_i(t_{1/2})\wh\bfV_i^{0,*})\wh\bfV_i^{0}\wh\bfQ_i^{\top} + O(h^2 + h\bar \varepsilon_r)
    \end{align*}
    where we have used $\Vert (\bfI-\wh\bfU_i^0\wh\bfU_i^{0,*})\bfK_i(t_{1/2})\Vert \leq ch^2$ according to \eqref{eq:le2_approxK} and $\wh\bfV_i^{0,*} = \wh\bfQ_i^{\top}\bigotimes_{j\neq i} \wh\bfU_j^{0,*}$ in the last equality. We therefore immediately have 
    \begin{align*}
        \Vert A(t_1) \bigtimes_{j\neq i} \wh\bfU_j^{0,*} \times_i \widetilde\bfU_i^* - \wt C_i^1\Vert \leq c(h^3 + h^2 \bar\varepsilon_r)\,.
    \end{align*}
    Lastly, for any sets of indices $\mathcal{J}\subset \{1,\cdots,d\}$ with $|\mathcal{J}| \geq 2$ and $\mathcal{I} := \{1,\cdots,d\}\backslash \mathcal{J}$, we have with $Y_{1/2} := \bar C(t_{1/2})\bigtimes_{\ell} \wh\bfU_{\ell}^0$, where $A(t_{1/2}) = Y_{1/2} + O(h\varepsilon_r + h^2)$
        \begin{align*}
    	\,& \Vert A(t_1) \bigtimes_{i\in \mathcal{I}} \wh\bfU_i^{0}\wh \bfU_i^{0,*}\bigtimes_{j\in \mathcal{J}} \widetilde\bfU_j\widetilde\bfU_j^{*}\Vert \\
     \leq\,& h\Vert F(t,Y_{1/2}) \bigtimes_{i\in \mathcal{I}} \wh\bfU_i^{0}\widehat \bfU_i^{0,*}\bigtimes_{j\in \mathcal{J}} \widetilde\bfU_j\widetilde\bfU_j^{*}\Vert + c_1(h^2\varepsilon_r + h^3)\\
     \leq\,& h\Vert M(t,Y_{1/2}) \bigtimes_{i\in \mathcal{I}} \wh\bfU_i^{0}\widehat \bfU_i^{0,*}\bigtimes_{j\in \mathcal{J}} \widetilde\bfU_j\widetilde\bfU_j^{*}\Vert + c_1(h^2\varepsilon_r + h^3) + c_2h \varepsilon_{2r}\,.
    \end{align*}
    By definition of $M$, we have for any $j,\ell \in\mathcal{I}$ that $M(Y_{1/2}) \times_j \widetilde\bfU_j \times_{\ell} \widetilde\bfU_{\ell} =0$.
     
 \end{proof}
Finally, we are able to prove a robust error bound \ref{thm:global_error_bound_Tucker} for our proposed tensor integrator. 

 \begin{proof}[of Theorem~\ref{thm:global_error_bound_Tucker}]
     The proof is an immediate consequence of Lemma~\ref{lemma:At_bound} and Lemma~\ref{lemma:A_UU_Y}. Therefore, we have for the local error 
     \begin{align*}
         ||A(t_1) - Y^1|| &= ||A(t_1) \pm A(t_1) \bigtimes_{i=1}^d \wh \bfU_i^1 \wh \bfU_i^{1,*} - Y_1|| \\
         &\leq ||A(t_1) - A(t_1) \bigtimes_{i=1}^d \wh \bfU_i^1 \wh \bfU_i^{1,*}|| + ||A(t_1) \bigtimes_{i=1}^d \wh \bfU_i^1 \wh \bfU_i^{1,*} - Y_1|| \\
         &\leq c_1 h^3 + c_2 \bar\varepsilon_{2r}h + c_3 \bar\varepsilon_r h^2 + C ( h^3 +  (\varepsilon_{r}+\bar\varepsilon_r) h^2 + (\bar\varepsilon_{2r}+\varepsilon_{2r})h ) \\
         &\leq c_1 h^3 + c_2 h^2 \varepsilon_r^{\mathrm{max}} + c_3 h\varepsilon_{2r}^{\mathrm{max}},
     \end{align*}
     where in last inequality the $c_i$ coefficients were newly set. The dependence of the error on $\delta$ is obtained in the standard way by comparing $Y^1$ with the solution of the perturbed initial value problem $\dot{\wt A} = F(t,\wt A)$, $\wt A(0) = Y^0$, and applying the triangle inequality together with Gronwall's inequality, which yields the additional term $e^{Lt^*}\delta = c_0 \delta$. An additional truncation with tolerance $\vartheta$ gives
     \begin{align*}
         ||A(t_1) - Y^1|| \leq c_0 \delta +  c_1 h^3 + c_2 h^2 \varepsilon_r^{\mathrm{max}} + c_3 h\varepsilon_{2r}^{\mathrm{max}} + c_4\vartheta.
     \end{align*}
     The stated result, a global error bound, is then obtained by a standard Lady-Windermere's argument, see e.g \cite[Sections I.7 and II.3]{HairerNorsettWanner:ODE_BOOK1}.
 \end{proof}

    \begin{remark}
        Our analysis, in particular Lemma~\ref{lemma:A_UU_Y}, shows that we can improve the first-order parallel BUG integrator of \cite{ceruti2024parallel}. This integrator updates blocks in the coefficient tensor as
        \begin{align*}
            \widehat C^1 \bigtimes_{j\neq i} (\bfI_{r_j}, \bm{0}_{r_j}) \times_i ( \bm{0}_{r_i}, \bfI_{r_i}) = hF(Y_0)\bigtimes_{j\neq i} \bfU_{j}^{0,*} \times_i \widetilde\bfU_i^{*} \in\C^{r_1\times\cdots\times r_d}\,.
        \end{align*}
        A disadvantage of this construction is that the update is based on an explicit Euler method, meaning that poor accuracy and stability can affect the integrator, even if implicit and high-order integration methods are used to update $\bfK_i(t)$ and $\bar C(t)$. Moreover, it requires additional operations such as evaluations of the right-hand side. Our analysis shows that this construction can be replaced by 
        \begin{align*}
            \widehat C^1 \bigtimes_{j\neq i} (\bfI_{r_j}, \bm{0}_{r_j}) \times_i ( \bm{0}_{r_i}, \bfI_{r_i}) = \text{Ten}_i(\widetilde\bfU_i^{*}\bfK_i(t_1)\bfQ_i^{\top})\,
        \end{align*}
        thus decoupling the update from an explicit Euler method and reducing computational costs. Note that we are using the notation from \cite{ceruti2024parallel} except for $\bfQ_i$, which is used as defined in Section~\ref{sec:parallelTucker} in the $\bfK_i$-step.
    \end{remark}

% \textcolor{blue}{
% Mögliche Dinge die man noch machen kann:
% \begin{itemize}
%     \item Norm- und Energieerhaltung für Schrödingersysteme beweisen. Solle ähnlich gehen wie im augmented BUG Fall.
%     \item Masse- und Energieerhaltung für kinetische Gleichungen beweisen. \textcolor{red}{Ich glaube das sollte sehr einfach gehen (man muesste den $K$-Schritt etwas anpassen). Das waere dann kinetik mit Unsicherheiten $\xi_i$ ($f(t,x,v,\xi_1,\dots,\xi_d)$), ohne Unsicherheiten hat man dann wieder eine Matrix ODE.}
%     \item Einen Remark schreiben, wie man den first order parallelen BUG stabiler und genauer machen kann, indem man die $\tilde C_i$ ersetzt durch $\text{Ten}_i(\widetilde\bfU_i^{*}\bfK_i(t_1)\wh\bfQ_i^{\top})$.
% \end{itemize}
% }

\section{Second-order augmented BUG for Tucker tensors}\label{sec:2nd_order_augmented_BUG}
The considerations and the proof of the error bound for the second-order parallel BUG indicate another integrator that shares the same robust second-order error bound. The integrator will be in the spirit of the (first-order) augmented BUG integrator \cite{CeKL22}. In this section, we formulate the second-order augmented BUG integrator and prove that the same robust second order error bound holds. We note that the proposed integrator differs from a Tucker extension of the midpoint BUG integrator introduced in \cite{ceruti2024robust}. Nevertheless, both methods are conceptually similar in that they rely on information of the basis at the midpoint.
\subsection{Formulation of the algorithm}
Given the rank $r = (r_1, \cdots, r_d)$ numerical solution in Tucker format $Y^0 = C^0 \bigtimes_{i=1}^d \bfU_i^0$ at time $t_0$, one step of the second order augmented Tucker integrator from $t_0$ to $t_1$ reads:
\begin{enumerate}
    \item \textbf{Pre-Augment}: Construct the augmented basis $\wh \bfU_i^0$, the matrices $\bfQ_i^0$, $\bfS_i^{0}$, $\wh \bfQ_i^0$ and the tensor $\wh C_F^0$ as in Sec.~\ref{subsec:2ndorder_parallel_Tucker_formulation} for $i=1,\dots,d$. 
    \item \textbf{$\bfK_i$-step}: Set $\wh\bfV_i^{0,*} = \wh\bfQ_i^{\top} \bigotimes_{j\neq i}^d \wh\bfU_j^{0,\top} \in\C^{2r_i \times n_{\lnot i}}$. Integrate the $n_i \times 2r_i$ matrix differential equations
		\begin{align*} 
		\dot{\textbf{K}}_i(t) = \bfF_i(t, \textbf{K}_i(t) \wh\bfV_i^{0,*}) \wh\bfV_i^0, \qquad \textbf{K}_i(t_0) = \wh\bfU_i^0\wh\bfU_i^{0,*}\bfU_i^0 \bfS_i^0 \bfQ_i^{\top}\overline{\widehat\bfQ_i}\,.
		\end{align*}
		Construct $\wh \bfU_i^1 = (\wh \bfU_i^0, \widetilde \bfU_i) \in \C^{n_i\times 4r_i}$ as an orthonormal basis of the range of the $n_i\times 4r_i$ matrix $(\wh\bfU_i^0, \bfK_i(t_1))$ (e.g.~by QR decomposition).
        \item \textbf{C-step}: From $t=t_0$ to $t_1$ integrate the $4r_1 \times \cdots \times 4r_d$ tensor differential equation
		\begin{align*} 
		\dot{\wh C}(t) = F(t, \wh C(t) \bigtimes_{i=1}^d\wh\bfU_{i}^{1}) \bigtimes_{i=1}^d\wh\bfU_i^{1, *}, 
		\qquad \wh  C(t_0) = C^0 \bigtimes_{i=1}^d \wh\bfU_i^{1,*}\bfU_i^{0}.
		\end{align*}
        \item \textbf{Truncate} according to a tolerance parameter $\vartheta$.
\end{enumerate}
Since step~$3$ is performed in the full $4r_1 \times \dots \times 4r_d$, no augmentation step is needed. However, this step becomes significantly more expensive than the $2r_1 \times \dots \times 2r_d$ tensor differential equation from the second-order parallel version in Sec.~\ref{subsec:2ndorder_parallel_Tucker_formulation}. 

\subsection{Robust error bound}
\begin{theorem}\label{thm:global_error_bound_Tucker_2nd_augmented}
     Under the assumptions \ref{ass:bounded}-\ref{ass:init}, the error of the numerical approximation $Y_n$ at time $t_n = t_0 + nh$, obtained with n time steps of the second-order augmented Tucker tensor integrator with step size $h>0$ and rank-truncation tolerance $\vartheta$, is bounded by
     \begin{align}
         ||Y_n - A(t_n)|| \leq c_0\delta +c_1 h^2 + c_2 h \varepsilon_r^{\mathrm{max}} + c_3 \varepsilon_{2r}^{\mathrm{max}} + c_4 n\vartheta,
     \end{align}
     where all appearing constants $c_0,\dots,c_4$ are independent of small singular values of matricizations of the connection tensors.
 \end{theorem}

\begin{proof}
First, the dependence of the error bound on the initial error $\delta$ is proven as in the proof of Theorem~\ref{thm:global_error_bound_Tucker}. Using Lemma~\ref{lemma:At_bound}, the remaining proof is a trivial reformulation of \cite[Theorem~2]{ceruti2024robust} and \cite[Section~4]{hauck2026highorderrobustbasisupdate} for tensor--valued solutions. We subdivide the proof into two parts (a) and (b). Our goal is to bound the local error assuming $\bfA(t_0) = \bfY_0$, resulting in a global error bound with Lady Windermere's fan which propagates errors through the exact flow. \\
(a) We start by bounding the residual
$$
R(t) := A(t) - A(t)\bigtimes_{i=1}^d \wh \bfU_i^1 \wh \bfU_i^{1,*}
$$
for $t\in[t_0,t_1]$. By Lemma~\ref{lemma:At_bound}, we have with $\mu=h(h^2+\bar\varepsilon_{2r} + \bar\varepsilon_r h)$ that
\begin{align*}
    R(t_1) = A(t_1) \bigtimes_{i=1}^d \wh \bfU_i^1 \wh \bfU_i^{1,*} - A(t_1) =  O(\mu)\,.
\end{align*}
Moreover, $R(t_0) = A(t_0) - A(t_0) \bigtimes_{i=1}^d \wh \bfU_i^1 \wh \bfU_i^{1,*} = 0$.
%Thus, using $A(t_0) - A(t_0) \bigtimes_{i=1}^m \wh \bfU_i^1 \wh \bfU_i^{1,*} = 0$, we have
%\begin{align*}
%    A(t_1) - A(t_0) - (A(t_1)- A(t_0)) \bigtimes_{i=1}^m \wh \bfU_i^1 \wh \bfU_i^{1,*} = O(\mu)\,.
%\end{align*}
%(b) We define the residual term
%$$
%R(t) := A(t) - A(t)\bigtimes_{i=1}^m \wh \bfU_i^1 \wh \bfU_i^{1,*},
%$$
%which by part (a) fulfills
%$$
%R(t_1)-R(t_0) = O(\mu).
%$$
%Moreover, as discussed in part (a), we have
%$$
%R(t_0) = 0  \quad\hbox{ and }\quad R(t_1)=O(\mu).
%$$
To bound $R$ at a general time $t$, we note that with bounded second derivatives we have for $ t_0\le t \le t_1$
\begin{align*}
    \Vert R(t)\Vert =\,& \Vert R(t_0) + (t-t_0)R'(t_0)\Vert + O(h^2)\\
    \leq\,& \Vert R(t_1)\Vert + O(h^2) = O(h^2 +\mu)\,.
\end{align*}
%$p$-th derivatives of $R$, we have for $p=2$
%\begin{align*}
%    \Vert R(t)\Vert =\,& \Vert R(t_0) + \sum_{k=1}^{p-1}R^{(k)}(t_0)\frac{(t-t_0)^k}{k!}\Vert + O(h^p)\\
%    \leq\,& \Vert R(t_0) + \sum_{k=1}^{p-1}R^{(k)}(t_0)\frac{h^k}{k!}\Vert + O(h^p)\\
%    \leq\,& \Vert R(t_1)\Vert+ O(h^p) + O(h^p) = O(h^p +\mu)\,.
%\end{align*}
%Therefore, we obtain the bound
%$$
%R(t) = O(h^2+\mu), \qquad t_0\le t \le t_1.
%$$
(b) We first write 
\begin{align*}
    \wh Y_1 - A(t_1) =\,& \wh Y_1 - A(t_1)\bigtimes_{i=1}^d \wh \bfU_i^1 \wh \bfU_i^{1,*} + A(t_1)\bigtimes_{i=1}^d \wh \bfU_i^1 \wh \bfU_i^{1,*} - A(t_1)\\
=\,&  (\wh C(t_1) - A(t_1)\bigtimes_{i=1}^d \wh \bfU_i^{1,*}) \bigtimes_{i=1}^d \wh \bfU_i^1 + R(t_1) \,.
\end{align*}
By part (a), $R(t_1) = O(\mu)$. Thus, it remains to show that
\begin{equation}\label{eq:S-tilde}
\wh C(t_1) - A(t_1)\bigtimes_{i=1}^d \wh \bfU_i^{1,*}=O(\mu).
\end{equation}
For this, let $t_0\le t \le t_1$, and define $ \wt C(t) := A(t)\bigtimes_{i=1}^d \wh \bfU_i^{1,*}$.
		We write
		\begin{equation*}
		\begin{aligned}
		A(t) 
		&=   A(t) - A(t)\bigtimes_{i=1}^d \wh \bfU_i^1 \wh \bfU_i^{1,*}
         + A(t)\bigtimes_{i=1}^d \wh \bfU_i^1 \wh \bfU_i^{1,*}
		=  R(t) + \wt C(t)\bigtimes_{i=1}^d \wh \bfU_i^1
		\end{aligned}
		\end{equation*}
        and
        \begin{equation*}		\begin{aligned}
		F(t, A(t)) 
		&= F(t, \wt C(t)\bigtimes_{i=1}^d \wh \bfU_i^1 + R(t) ) 
         = F(t, \wt C(t)\bigtimes_{i=1}^d \wh \bfU_i^1) + D(t)
		\end{aligned}
		\end{equation*}
		where the defect $D$ is defined as
		$
		D(t) := F(t, \wt C(t)\bigtimes_{i=1}^d \wh \bfU_i^1 + R(t)) - F(t, \wt C(t)\bigtimes_{i=1}^d \wh \bfU_i^1).
		$
		Since $F$ is Lipschitz continuous with constant $L$ and $R(t)$ has a sufficiently small bound as shown in part (a), the defect fulfills the bound
		$$
		\|  D(t) \| \le L \| R(t) \| =O(h^2+\mu).
		$$
		Then, we compare
		\begin{equation*}
		\begin{aligned}
		&\dt{\wh C}(t) = F(t,\wh C(t)\bigtimes_{i=1}^d \wh \bfU_i^1)\bigtimes_{i=1}^d \wh \bfU_i^{1,*}\,, 
		\qquad
		&\wh C(t_0) = Y_0 \bigtimes_{i=1}^d \wh \bfU_i^{1,*},
        \\
        &\dt{\wt C}(t) = F(t, \wt C(t)\bigtimes_{i=1}^d \wh \bfU_i^1)\bigtimes_{i=1}^d \wh \bfU_i^{1,*} + D(t)\bigtimes_{i=1}^d \wh \bfU_i^{1,*}, 
		\qquad
		&\wt C(t_0) = Y_0 \bigtimes_{i=1}^d \wh \bfU_i^{1,*}.
		\end{aligned}
		\end{equation*}
		Lastly, we obtain with the Gronwall inequality
		$$
		\|  \wh C(t_1) - \wt C(t_1) \| 
		\leq \int_{t_0}^{t_1} e^{L(t_1-s)} \, \| D(s) \| \, ds
		=O(h(h^2+\mu))= O(\mu).
		$$
		Thus, \eqref{eq:S-tilde} holds, which yields
        \begin{align*}
            A(t_1) - Y_1 = A(t_1) - \wh Y_1 + \wh Y_1 - Y_1 = O(\mu + \vartheta)\,,
        \end{align*}
        concluding the proof. 
\end{proof}

\begin{remark}
    We note that this proof can be trivially extended to arbitrary order if $A(t_1) \bigtimes_{i=1}^d \wh \bfU_i^1 \wh \bfU_i^{1,*} = A(t_1) + O(\mu_p)$ for $\mu_p := h^{p+1} + h\bar\varepsilon_r$. Here, one needs to replace $\mu$ by $\mu_p$ in the proof and use order $p$ derivatives of $R(t_0)$. This provides a strategy for constructing higher-order augmented BUG integrators. 
\end{remark}

\subsection{Properties}\label{subsec:properties_2nd_augment}
% \ds{@Jonas: Würdest du das alles aufnehmen ins paper? Also nur die Theoreme staten und auf die rank-adaptive Arbeit verweisen für die Beweise? Und wenn ja, nur hier und dann im TTN Fall nur schreiben ''gilt immer noch''?}\jknote{Ich wuerde es sehr kurz halten und einfach sagen, dass der Beweis direkt vom aug. BUG Paper folgt.}

% Der second-order augmented BUG sollte Norm und Energie erhaltend sein, der parallele nicht. Der Beweis für Norm und Energieerhaltung sollte der gleiche sein wie im augmented BUG paper, ebenso auch der Beweis für die Gradient system Eigenschaft sollte gelten. -> Lemma 1 aus rank-adaptive TTN paper muss gelten (und tut es auch).

The first-order augmented BUG integrator from \cite{CeKL22} shares desirable properties such as norm and energy preservation for Schrödinger equations and energy diminishing for gradient systems, all up to truncation. The proofs of these properties rely mainly on the fact that we have an orthogonal basis and perform a Galerkin step for the update of the connecting tensor. The proposed second-order augmented BUG integrator relies on the same facts with the only difference that the (orthogonal) basis is larger compared to the one of the first-order augmented BUG. Hence, the proofs for norm and energy preservation and energy diminishing for gradient systems can be directly applied to the proposed second-order augmented BUG integrator. We therefore only state the properties and refer to \cite{CeKL22} (for the matrix case) and to \cite{CLS23} (for the tensor case) for the proofs.
\paragraph{Norm conservation}
\begin{theorem}
    Suppose that the function $F$ satisfies 
    \begin{align*}
        \text{Re} \langle Y,F(Y)\rangle = 0 \quad \text{for all} \quad Y \in \C^{n_1 \times \dots \times n_d},
    \end{align*}
    where $\langle \cdot,\cdot \rangle$ denotes the Euclidean inner product. Then a step of the second-order augmented BUG integrator preserves the norm: for every step size $h>0$,
    \begin{align*}
        \|\wh Y^1\| = \|Y^0\|.
    \end{align*}
    This implies a near-conservation of norm up to a multiple of the truncation tolerance $\vartheta$:
    \begin{align*}
        \big\vert \|Y^1\| - \|Y^0\| \big\vert  \leq c \vartheta,
    \end{align*}
    with $c = \|C^0\|(d-1) + 1$.
\end{theorem}
\paragraph{Energy conservation for Schrödinger equations}
Consider the tensor Schrödinger equation with a Hamiltonian $H:\C^{n_1 \times \dots \times n_d} \rightarrow \C^{n_1 \times \dots \times n_d}$ that is linear and self-adjoint. Further, define the energy by 
\begin{align*}
    E(Y) = \langle Y,H[Y]\rangle.
\end{align*}
\begin{theorem}
     The second-order augmented BUG integrator preserves the energy: for every step size $h>0$,
    \begin{align*}
        E(\wh Y^1) = E(Y^0).
    \end{align*}
    This implies a near-conservation of energy up to a multiple of the truncation tolerance $\vartheta$:
    \begin{align*}
        \big \vert E(Y^1) - E(Y^0) \big\vert \leq c \vartheta \|H[Y^1 + \wh Y^1]\|,
    \end{align*}
    with $c = \|C^0\|(d-1) + 1$.
\end{theorem}
\paragraph{Energy decay for gradient systems}
Now consider a gradient system
\begin{align*}
    \dot A(t) = -\nabla E(A(t)) \quad \text{with a given function} \ E:\R^{n_1 \times \dots \times n_d} \rightarrow \R.
\end{align*}
Along every solution, we have energy decay
\begin{align*}
    \partial_t E(A(t)) = \langle \nabla E(A(t)),\dot A(t) \rangle = - \|\nabla E(A(t))\|^2.
\end{align*}

\begin{theorem}
    The second-order augmented BUG integrator diminishes the energy: for every step size $h>0$,
    \begin{align*}
        E(\wh Y^1) \leq E(Y^0) - \alpha^2 h,
    \end{align*}
    where $\alpha =  \min_{0 \leq \mu \leq 1} \| \nabla E(\wh{Y}(t_0 + \mu h)) \bigtimes_{i=1}^{d} \wh{\bfU}_{i}^{1,*} \| = \| \nabla E({Y}^0) \bigtimes_{i=1}^{d} \wh{\bfU}_{i}^{1,*} \| + O(h)$. By the mean value theorem and the truncation, this implies:
    \begin{align*}
	E(Y^1 ) \leq E(\wh{Y}^1 ) + \beta c \vartheta \le E(Y^0)  - \alpha^2 h + \beta c \vartheta
\end{align*}
with $\beta = \text{max}_{0\leq \mu \leq 1} \| \nabla E( \mu Y^1  + (1-\mu)\wh{Y}^1 ) \|=\| \nabla E(Y^1) \|+O(\vartheta)$.
\end{theorem}

\section{Second-Order Tree Tensor Network Integrators}\label{sec:ttn}
Tree tensor networks, also known as hierarchical Tucker tensors, are a hierarchical, data-sparse format to approximate high-dimensional tensors. We briefly summarize the definition of tree tensor networks following the notation of \cite{CeLW20,CLS23,ceruti2024parallel}. The hierarchical structure is encoded in a tree with unequal leaves defined as follows.
\begin{definition}[Ordered trees with unequal leaves]
	
	Let $\mathcal{L}$ be a given finite set, the elements of which are referred to as leaves.
	We  define the set $\mathcal{T}$ of trees  $\tau$ with the corresponding set of leaves $L(\tau)\subseteq \mathcal{L}$ recursively as follows:
	
	\begin{enumerate}[label=(\roman*)]
		\item
		 {\em Leaves are trees:} 
		$ \mathcal{L} \subset \mathcal{T}$,\ \text{ and }\  $L(l) := \{l\}$ for each $l \in \mathcal{L}$.
		\item  {\em Ordered $m$-tuples of trees with different leaves are again trees:} 
		If, for some $m\ge 2$,
		$$
		\tau_1, \dots, \tau_m \in \mathcal{T} 
		\quad \text{ with }\quad
		L(\tau_i ) \cap L(\tau_j ) = \emptyset \quad \forall i \neq j,
		$$
		 then their ordered $m$-tuple is in $\mathcal{T}$:		
		$$ \tau := (\tau_1, \dots, \tau_m) \in \mathcal{T} 
		, \quad \text{ and }\quad 
		L(\tau) := \dot{\bigcup}_{i=1}^m L(\tau_i) \ . 
		$$
	\end{enumerate}		
	%An element $\tau \in \mathcal{T}$ is called a tree.
\end{definition}
The trees $\tau_1,\dots,\tau_m$ are called direct subtrees of $\tau$, together with direct subtrees of those direct subtrees are called subtrees of $\tau$. The definition allows for a partial ordering. For $\tau,\sigma \in \mathcal{T}$ we set
\begin{align*}
    &\sigma \leq \tau \text{ if and only if } \sigma \text{ is a subtree of } \tau, \\
    &\sigma < \tau \text{ if and only if } \sigma \leq \tau \text{ and } \sigma \neq \tau.
\end{align*}
Following \cite{CeLW20,CLS23}, we define the height of a tree $\tau$ by $0$ if $\tau = \ell$ is a leaf. For a tree $\tau = (\tau_1,\dots,\tau_m)$, we set $h(\tau) := \max \{h(\tau_1),\dots,h(\tau_m)\} +1$.

Now fix a maximal tree $\bar \tau$. With each leaf we associate a basis matrix $\bfU_\ell \in \C^{n_\ell \times r_\ell }$ for $\ell \in \mathcal{L}$ and with each tree $\tau=(\tau_1,\dots,\tau_m) \leq \bar \tau$ we associate a connecting tensor $C_\tau \in \C^{r_\tau \times r_{\tau_1} \times \dots \times r_{\tau_m} }$, where we assume $r_{\bar \tau} = 1$. We assume connecting tensors of full multilinear rank $(r_\tau,r_{\tau_1},\dots,r_{\tau_m})$, which implies the necessary condition
\begin{align*}
    r_{\tau_i} \leq \prod_{j=0,j \neq i}^m r_{\tau_j} \quad \text{for all } i=0,1,\dots,m,
\end{align*}
where we used $r_{\tau_0} = r_{\tau}$. With all this, we define a tree tensor network recursively from the bottom to the top.

\begin{definition}[Tree tensor network]
	For a given tree $\bar \tau \in \mathcal{T}$ and basis matrices $\bfU_l$ and connection tensors $C_\tau$ as described above, we recursively define a tensor $Y_{\bar \tau}$ with a tree tensor network representation (or briefly a tree tensor network) as follows: 
	\begin{enumerate}[label=(\roman*)]
		\item
		For each leaf  $\,\tau = l \in \mathcal{L}$, we set 
		$$Y_l := \bfU_l^\top \in \C^{r_l \times n_l} \ .  $$
		\item
		For each subtree $\tau = ( \tau_1, \dots, \tau_m)$  (for some $m\ge 2$) of $\bar\tau $, 
		we set \\
		$n_\tau = \prod_{i=1}^m  n_{\tau_i}$ and $\I_\tau$ the identity matrix of dimension $r_\tau$, and
		\begin{align*}
		& Y_\tau := C_\tau \times_0 \bfI_{\tau} \bigtimes_{i=1}^m \bfU_{\tau_i} 
		\in \mathbb{C}^{r_\tau \times n_{\tau_1} \times \dots \times n_{\tau_m}},
		\\
		& \bfU_{\tau} := \mat_0( Y_\tau )^\top \in \mathbb{C}^{n_\tau \times r_\tau} \ .
		\end{align*}
		The subscript $0$ in $\times_0$ and $\mat_0( Y_\tau )$ refers to the mode $0$ of dimension $r_{\tau}$ in $\C^{r_\tau \times r_{\tau_1}\times\dots\times r_{\tau_m}} $. 
	\end{enumerate}
	The tree tensor network $Y_{\bar{\tau}}$ (more precisely, its representation in terms of the matrices $\bfU_\tau$) is called {\em orthonormal} if for each subtree $\tau < \bar{\tau}$, the matrix $\bfU_\tau$ has orthonormal columns. 
\end{definition}
Note that tree tensor networks do not allow for loops in the network. Binary tree tensor networks have been studied in the mathematical literature as \textit{hierarchical Tucker tensors} \cite{hackbusch_book2012} and for general tree structures as \textit{tensors in tree-based format} \cite{falco2021tree}. \textit{Tensor trains} \cite{oseledets2011tensor}, also known as \textit{matrix product states (MPS)} \cite{MPS_paper} in the physics literature, are contained in the class of tree tensor networks.

In the following paragraph, we briefly summarize the construction of reduced versions of the function $F_{\bar \tau}$ and the initial data $Y_{\bar \tau}^0$. We refer to \cite{CeLW20} for a detailed construction via prolongation and restriction operators. Assume that $F_{\bar \tau}$ maps tensors in tree tensor network representation to the same tensor type. Further, suppose a tree $\tau = (\tau_1,\dots,\tau_m) \leq \bar \tau$ and a corresponding function $F_\tau$. For the tree $\tau$ we define the tensor space $\mathcal{V}_\tau = \C^{r_\tau \times n_{\tau_1} \times \dots \times n_{\tau_m}}$ and the manifold of tree tensor networks of full tree rank $(r_\sigma)_{\sigma \leq \tau}$, which we denote by $\mathcal{M}_\tau = \mathcal{M}_\tau (\tau, (n_\ell)_{\ell \in \mathcal{L}(\tau)},(r_\sigma)_{\sigma \leq \tau})\subset \mathcal{V}_\tau$. 

By induction, we assume that $F_\tau: \mathcal{V}_\tau \rightarrow \mathcal{V}_\tau$ and $Y_\tau^0 \in \mathcal{M}_\tau$ are already constructed. For the construction of the reduced function $F_{\tau_i}$ and reduced initial data $Y_{\tau_i}^0$, we consider $Y_\tau^0 = C_\tau^0 \times_0 \bfI_\tau \bigtimes_{j=1}^m \bfU_{\tau_j}^0$ and define the matrix
\begin{align*}
    \bfV_{\tau_i}^0 = \mat_i \left( \ten_i(\wt\bfQ_{\tau_i}^{0,\top}) \times_0 \bfI_\tau \bigtimes_{j \neq i } \bfU_{\tau_j}^0 \right)^*,
\end{align*}
where $\wt\bfQ_{\tau_i}^0$ is the unitary factor in the QR decomposition $\mat_i (C_\tau^0)^\top = \wt\bfQ_{\tau_i}^0 \wt \bfR_{\tau_i}^0$. The prolongation $\pi_{\tau,i}: \mathcal{V}_{\tau_i} \rightarrow \mathcal{V}_\tau$ and restriction $\pi_{\tau,i}^\dagger: \mathcal{V}_\tau \rightarrow \mathcal{V}_{\tau_i}$ are defined by
\begin{align*}
    \pi_{\tau,i}(Y_{\tau_i}) &:= \ten_i \left( \mat_0 (Y_{\tau_i})^\top \bfV_{\tau_i}^{0,*} \right) \in \mathcal{V}_{\tau}, \ \text{ for } Y_{\tau_i} \in \mathcal{V}_{\tau_i}, \\
    \pi_{\tau,i}^\dagger (Z_\tau ) &:= \ten_0 \left( (\mat_i(Z_\tau)\bfV_{\tau_i}^0)^\top \right) \in \mathcal{V}_{\tau_i}, \ \text{ for } Z_\tau \in \mathcal{V}_\tau.
\end{align*}
Note that $\pi_{\tau,i}^\dagger$ is both, the left inverse and the adjoint of $\pi_{\tau,i}$, cf. \cite{CeLW20}. Then, for a given function $F_{\bar \tau} = F: \mathcal{V}_{\bar \tau} \rightarrow \mathcal{V}_{\bar \tau}$ and a tree tensor network $Y_{\bar \tau}^0$, we recursively define for each tree $\tau = (\tau_1,\dots,\tau_m) \leq \bar \tau$
\begin{align*}
    F_{\tau_i} &= \pi_{\tau,i}^\dagger \circ F_\tau \circ \pi_{\tau,i}, \\
    Y_{\tau_i}^0 &= \pi_{\tau,i}^\dagger (Y_\tau^0),
\end{align*}
where $\pi_{\tau,i}^\dagger$ and $\pi_{\tau,i}$ are the restriction and prolongation, respectively. 

\subsection{Second-order parallel TTN integrator}
Fix a tree $\bar \tau$ together with a corresponding tree tensor network $Y_{\bar \tau}^0 = C_{\bar \tau}^0 \times_0 \bfI_{\bar \tau} \bigtimes_{i=1}^m \bfU_{\bar \tau_i}^0$. To improve readability, we will assume in the following that the ranks of $\wh \bfU_\tau^0$ and $\wh \bfQ_\tau^0$ are always equal to $2r_\tau^0$ for all $\tau \leq \bar \tau$, which is in general not true. However, everything we write below also works for the other case where the actual rank is $\leq 2r_\tau^0$ without any further adjustments. \\
We formulate the algorithm for tree tensor networks. In the Tucker case, this formulation falls back to the integrator presented in Sec.~\ref{sec:parallelTucker}. Analogously, we perform the following steps:
\begin{enumerate}
    \item \textbf{Pre-augmentation} (sequential, leaves to root): augment all
    bases to obtain $\wh\bfU_\tau^0$ and the pre-augmented network
    $\wh Y_{\bar\tau}^0$ of tree ranks $(2r_\tau)_{\tau \leq \bar \tau}$.
    \item \textbf{Construction of the reductions} (sequential, root to leaves):
    compute $\bfQ_{\tau}^0$, $\bfS_{\tau}^0$ and $\wh\bfQ_{\tau}^0$ for every
    subtree $\tau \leq \bar \tau$ and the augmented reduced functions $\wh F_{\tau}$ and
    initial data $\wh Y_{\tau}^0$.
    \item \textbf{Time integration} (fully parallel): for each $\tau \leq \bar \tau$ solve all the resulting effective evolution equations fully in parallel.
    \item \textbf{Augmentation} (sequential, leaves to root): combine the evolved
    quantities with the initial data into a network of tree ranks $(4r_\tau)_{\tau \leq \bar \tau}$.
    \item \textbf{Truncation}: truncate back to tree ranks $(r_\tau^1)_{\tau \leq \bar \tau}$ with tolerance $\vartheta$.
\end{enumerate}
These five steps are described in detail in the following subsections. 
%\subsubsection{Pre-augmentation: the bases $\wh \bfU_\tau^0$} \label{subsub:pre_augment_TTN}
\subsubsection{\texorpdfstring%
  {Pre-augmentation: $\wh \bfU_{\tau}^0$ and $\wh Y_{\bar \tau}^0$}%
  {Pre-augmentation}} \label{subsub:pre_augment_TTN}
The pre-augmentation step is performed from the leaves to the root. On each level of the tree, the pre-augmentation can be done fully in parallel. 

\textbf{Augmentation of basis matrices:} Suppose $\tau_i = \ell$ is a leaf with corresponding basis matrix $\bfU_{\ell}^0 \in \C^{n_\ell \times r_\ell^0}$. Then we define the augmented $\wh \bfU_{\ell}^0$ by 
\begin{align*}
    \wh \bfU_{\ell}^0 := \text{orth} \left( [\bfU_{\ell}^0,\bfF_{\ell}(\bfY_{\ell}^0)] \right) \in \C^{n_\ell \times 2r_\ell^0}.
\end{align*}
Note that the matrix $\bfV_{\tau_i}^0$ is implicitly included in the definition of $\bfF_{\ell}$. For a later recursion in the definition, we define $\bfM_{\tau_i} := \wh\bfU_{\tau_i}^{0,*} \bfU_{\tau_i}^{0}$. 

\textbf{Augmentation of core tensors:} Suppose a tree $\tau = (\tau_1,\dots,\tau_m) \leq \bar \tau$, where we have computed already $\wh \bfU_{\tau_i}^0$ for $i=1,\dots,m$. In the same logic as for the leaves, we would like to define  
\begin{align*}
    \wh \bfU_{\tau}^0 = \text{orth} \left( [\bfU_{\tau}^0, \mat_0(F_{\tau}(Y_{\tau}^0) )^\top] \right).
\end{align*}
However, due to its large dimensions, this matrix is infeasible to compute and store. Therefore, we represent $\wh \bfU_{\tau}^0$ as a Tucker tensor of the form $\wh \bfU_{\tau}^0 = \mat_0(\wh Y_{\tau}^0)^{\top}$, where $\wh Y_{\tau}^0 := \wh C_{\tau}^0 \times_0 \bfI_\tau \bigtimes_{i=1}^m \wh\bfU_{\tau_i}^0$ for an unknown tensor $\wh C_{\tau}^0$ that remains to be determined. This tensor can be computed by testing $
\bfU_{\tau}^0$ against $\bigotimes_{i=1}^m \wh \bfU_{\tau_i}^{0}$. Let us start with only the first $r_{\tau}^0$ columns of $\wh \bfU_{\tau}^0$. Then, we obtain by the unfolding formula \eqref{eq:unfolding_formula}
\begin{align*}
    \left(\bigotimes_{i=1}^m \wh \bfU_{\tau_i}^{0,*} \right) \bfU_{\tau}^0 &= \bigotimes_{i=1}^m \wh \bfU_{\tau_i}^{0,*} \mat_0(Y_{\tau}^0)^\top = \bigotimes_{i=1}^m \wh\bfU_{\tau_i}^{0,*} \bfU_{\tau_i}^{0}\mat_0(C_{\tau}^0)^\top \\
    &= \bigotimes_{i=1}^m \bfM_{\tau_i} \mat_0(C_{\tau}^0)^\top = \mat_0(C_{\tau}^0 \bigtimes_{i=1}^m \bfM_{\tau_i})^\top.
\end{align*}
The remaining $r_{\tau}^0$ columns of $\wh \bfU_{\tau}^0$ span the information in $\mat_0(F_{\tau}(Y_{\tau}^0))^\top$ that is not spanned by $\bfU_{\tau}^0$. Thus, we test this term against $\bigotimes_{i=1}^m \wh \bfU_{\tau_i}^{0}$ and obtain, by the unfolding formula \eqref{eq:unfolding_formula},
\begin{align*}
    \left(\bigotimes_{i=1}^m \wh \bfU_{\tau_i}^{0,*} \right) \mat_0(F_{\tau}(Y_{\tau}^0))^\top = \mat_0(F_{\tau}(Y_{\tau}^0) \bigtimes_{i=1}^m \wh \bfU_{\tau_i}^{0,*})^\top. 
\end{align*}
Since both matrices are now of smaller size we set 
\begin{align*}
    \mat_0(\wh C_{\tau}^0)^\top :=& \ \text{orth} \left( \bigotimes_{i=1}^m \wh \bfU_{\tau_i}^{0,*} \ [\bfU_{\tau}^0, \mat_0(F_{\tau}(Y_{\tau}^0))^\top] \right) \\ =& \ \text{orth} \left( [\mat_0(C_{\tau}^0 \bigtimes_{i=1}^m \bfM_{\tau_i})^\top, \mat_0(F_{\tau}(Y_{\tau}^0) \bigtimes_{i=1}^m \wh \bfU_{\tau_i}^{0,*} )^\top] \right) \in \C^{ \lnot 2r_{\tau}^0 \times 2r_{\tau}^0}.
\end{align*}
Since $C_\tau^0 \bigtimes_{i=1}^m \bfM_{\tau_i}$ is orthogonal, the  orthogonalization is computed such that the first $r_{\tau}^0$ columns still equal $C_\tau^0 \bigtimes_{i=1}^m \bfM_{\tau_i}$.
Finally, we recursively set 
\begin{align*}
    \wh Y_\tau^0 = \wh C_\tau^0 \times_0 \bfI_\tau \bigtimes_{i=1}^m \wh \bfU_{\tau_i}^0 \quad \text{and} \quad \wh \bfU_\tau^0 = \mat_0(\wh Y_\tau^0)^\top.
\end{align*}
Further, we define $\bfM_\tau := \wh \bfU_\tau^{0,*} \bfU_\tau^0$, which can be recursively computed from the leaves to the node $\tau$. 

%In other words, $ \wh Y_\tau^0$ is the original $Y_\tau^0$ projected in the pre-augmented basis, since 
%\begin{align*}
%    \wh Y_\tau^0 &= \wh C_\tau^0 \times_0 \bfI \bigtimes_{i=1}^m \wh \bfU_{\tau_i}^0 = C_\tau^0 \times_0 \bfI \bigtimes_{i=1}^m \wh \bfU_{\tau_i}^0 \bfM_{\tau_i} = C_\tau^0 \times_0 \bfI \bigtimes_{i=1}^m \wh \bfU_{\tau_i}^0 \wh \bfU_{\tau_i}^{0,*} \bfU_{\tau_i}^0.
%\end{align*}

%\subsubsection{Augmented restriction to the subtrees: $\bfQ_{\tau_i}$, $\bfS_{\tau_i}$ and $\wh \bfQ_{\tau_i}$}\label{subsub:constr_Q_TTN}
\subsubsection{\texorpdfstring%
  {Augmented restriction to the subtrees: $\bfQ_{\tau_i}^0$, $\bfS_{\tau_i}^0$ and $\wh \bfQ_{\tau_i}^0$}%
  {Augmented restriction to the subtrees}}\label{subsub:constr_Q_TTN}
As in the Tucker case, the reduced functions and the initial data on the subtrees are defined through augmented environment matrices $\wh \bfV_\tau^0$, which in turn are built from the factors $\bfQ_{\tau_i}^0$, $\bfS_{\tau_i}^0$ and $\wh \bfQ_{\tau_i}^0$. The tree structure introduces one genuinely new aspect compared to the Tucker
case. There, the only interior node is the root, whose $0$-mode is trivial, so
that the augmentation acts on the subtree modes only. For a subtree
$\tau < \bar\tau$, however, the $0$-mode carries the coordinates with respect to
the environment of the parent node, and this environment has to be augmented as
well. This is the only reason for distinguishing the augmented reduced functions $\wh F_\tau $ from the reduced functions $F_\tau$ from above - the two differ precisely in the dimension of the $0$-mode. Similarly, for the augmented initial data $\wh Y_\tau^0$. The construction of $\bfQ_{\tau_i}^0$, $\bfS_{\tau_i}^0$ and $\wh \bfQ_{\tau_i}^0$ and the corresponding augmented reduced function $\wh F_{\tau_i}$ and augmented initial data $\wh Y_{\tau_i}^0$ is the subject of this section and proceeds recursively from the root to the leaves.

%\jknote{Is this a good title? I got confused since this seems similar to what we did in the previous subsubsection. But I guess this is now the augmentation of the core tensors (and thus the test functions $V$), not of the basis. But since the core tensors can be interpreted as factors of the basis, they are quite similar.} \ds{Ich verstehe was du meinst. Grundsätzlich hast du auch recht, dass es hier "nur" drum geht, die Definitionen korrekt für den TTN Fall aufzuschreiebn. Ich würde aber sagen da alles mit korrekten Dimensionen aufzuschreiben ist nicht trivial.}\jknote{Ich stimme absolut zu, dass der Aufschrieb nicht trivial und sehr wichtig ist. Der Kommentar war nur: "Können wir das mit einem besseren Titel leichter verständlich machen oder mit einem einleitenden Satz erklären warum wir nicht die Berechnung von $Q$ etc schon im vorherigen Kapitel abgedeckt haben?"} \ds{Sowas von der Art her?}

For $\tau = \bar \tau$, we set $\wh{\mathcal{V}}_{\bar \tau} = \mathcal{V}_{\bar \tau}$. For every subtree $\tau < \bar\tau$, we introduce the augmented tensor space
\begin{align*}
    \wh{\mathcal{V}}_\tau := \C^{2r_\tau^0 \times n_{\tau_1} \times \cdots \times n_{\tau_m}},
\end{align*}
where $2r_\tau^0$ denotes the number of columns of the matrix $\wh\bfQ_\tau^0$ constructed below for the edge connecting $\tau$ to its parent. %For , $\hat r_\tau \leq 2 r_\tau^0$ denotes the number of columns of the pre-augmented basis matrix $\wh\bfU_\tau^0$.
 
In a recursion from the root to the leaves, we assume by induction that the augmented function $\wh F_\tau : [0,t^*] \times \wh{\mathcal{V}}_\tau \rightarrow
\wh{\mathcal{V}}_\tau$ and the embedded initial data
$\wh Y_\tau^0 \in \wh{\mathcal{V}}_\tau$ are already constructed, where at the
root we set $\wh F_{\bar\tau} = F_{\bar \tau}$ and $\wh  X_{\bar \tau}^0 =  Y_{\bar\tau}^0 $. For all subtrees $\tau < \bar \tau$ we assume the form 
\begin{align*}
    \wh X_\tau^0 = \wh D_\tau^0 \times_0 \bfI_{\tau}
    \bigtimes_{j=1}^m \bfU_{\tau_j}^0,
    \qquad
    \wh D_\tau^0 \in \C^{2r_\tau^0 \times  r_{\tau_1}^0 \times \cdots \times  r_{\tau_m}^0},
\end{align*}
i.e., the connecting tensor $\wh D_\tau^0$ carries the coefficients of $Y_\tau^0$ with respect to the augmented basis of the parent edge. In addition, we define the right-hand side projected in the pre-augmented basis
\begin{align*}
    \wh C_{F_\tau}^0
    := \wh F_\tau(t_0, \wh X_\tau^0) \times_0 \bfI_{\tau}
       \bigtimes_{j=1}^m \wh\bfU_{\tau_j}^{0,*}
    \;\in\; \C^{2r_\tau^0 \times 2 r_{\tau_1}^0 \times \cdots \times 2 r_{\tau_m}^0}.
\end{align*}

For the construction of the augmented reduced function $\wh F_{\tau_i}$ and
the embedded $\wh X_{\tau_i}^0$ of a subtree
$\tau_i$, $i = 1,\dots,m$, we first compute the QR decomposition
\begin{align*}
    \bfQ_{\tau_i}^0 \, \bfS_{\tau_i}^{0,\top} := \mat_i \bigl( \wh D_\tau^0 \bigtimes_{j\neq i}^m \wh \bfU_
    {\tau_j}^{0,*}  \bfU_{\tau_j}^0 \bigr)^{\top},
\end{align*}
with orthonormal $\bfQ_{\tau_i}^0 \in \C^{2 r_\tau^0 2 r_{\lnot \tau_i}^0 \times r_{\tau_i}^0}$ and $\bfS_{\tau_i}^{0,\top} \in \C^{r_{\tau_i}^0 \times r_{\tau_i}^0}$. Analogously to the Tucker case, we then define 
\begin{align}\label{eq:Qhat-construction}
    \wh\bfQ_{\tau_i}^0
    := \text{orth} \left[ \, \bfQ_{\tau_i}^0, \;
       \mat_i \bigl( \wh C_{F_\tau}^0 \times_i \bfU_{\tau_i}^{0,*}\wh \bfU_{\tau_i}^0 \bigr)^{\!\top}
       \right] 
    \;\in\; \C^{2r_\tau^0 2 r_{\lnot \tau_i}^0 \times 2r_{\tau_i}^0}.
\end{align}
Hence, we define the matrix
\begin{align*}
    \wh\bfV_{\tau_i}^0
    := \mat_i \Bigl( \ten_i \bigl( \wh\bfQ_{\tau_i}^{0,\top} \bigr)
      \times_0 \bfI_{\tau}
      \bigtimes_{j \neq i} \wh\bfU_{\tau_j}^0 \Bigr)^{\!*}
\end{align*}
and the augmented prolongation
$\wh\pi_{\tau,i} : \wh{\mathcal{V}}_{\tau_i} \rightarrow \wh{\mathcal{V}}_\tau$
and restriction
$\wh\pi_{\tau,i}^\dagger : \wh{\mathcal{V}}_\tau \rightarrow \wh{\mathcal{V}}_{\tau_i}$
by
\begin{align*}
    \wh\pi_{\tau,i}(Y_{\tau_i})
    &:= \ten_i \bigl( \mat_0 (Y_{\tau_i})^\top \, \wh\bfV_{\tau_i}^{0,*} \bigr)
    \in \wh{\mathcal{V}}_{\tau},
    \quad \text{for } Y_{\tau_i} \in \wh{\mathcal{V}}_{\tau_i},
    \\
    \wh\pi_{\tau,i}^\dagger (Z_\tau)
    &:= \ten_0 \bigl( ( \mat_i(Z_\tau) \, \wh\bfV_{\tau_i}^0 )^\top \bigr)
    \in \wh{\mathcal{V}}_{\tau_i},
     \quad\text{for } Z_\tau \in \wh{\mathcal{V}}_\tau.
\end{align*}
Since $\wh\bfQ_{\tau_i}$ and the $\wh\bfU_{\tau_j}^0$ have orthonormal columns,
the matrix $\wh\bfV_{\tau_i}^0$ has orthonormal columns as well, and
$\wh\pi_{\tau,i}^\dagger$ is both the left inverse and the adjoint of
$\wh\pi_{\tau,i}$. 
Finally, we recursively define
\begin{align*}
    \wh F_{\tau_i} &= \wh\pi_{\tau,i}^\dagger \circ \wh F_\tau \circ \wh\pi_{\tau,i},
    \\
    \wh X_{\tau_i}^0 &= \wh\pi_{\tau,i}^\dagger \bigl( \wh X_\tau^0 \bigr). 
\end{align*}
It remains to construct the connecting tensor $\wh D_{\tau_i}^0$ of $\wh X_{\tau_i}^0 = \wh\pi_{\tau,i}^\dagger(\wh X_\tau^0)$. We determine this tensor by computing
\begin{align*}
    \wh\pi_{\tau,i}^\dagger \bigl( \wh X_\tau^0 \bigr) &= \ten_0 \left( \wh \bfV_{\tau_i}^{0,\top} \mat_i (\wh X_\tau^0)^\top \right) \\
    &= \ten_0 \left( \overline{\mat_i( \ten_i(\wh \bfQ_{\tau_i}^{0,\top}) \times_0 \bfI_{\tau} \bigtimes_{j \neq i} \wh \bfU_{\tau_j}^{0,\top})} \ \bigotimes_{j \neq i}  \bfU_{\tau_j}^0 \mat_i(\wh D_\tau^0)^\top \bfU_{\tau_i}^{0,\top} \right) \\
    &= \ten_0 \left( \wh \bfQ_{\tau_i}^{0,*} \bigtimes_{j \neq i} \wh \bfU_{\tau_j}^{0,*} \bfU_{\tau_j}^0 \mat_i(\wh D_\tau^0)^\top \bfU_{\tau_i}^{0,\top} \right) \\
    &= \ten_0 \left( \wh \bfQ_{\tau_i}^{0,*} \mat_i \bigl( \wh D_\tau^0 \bigtimes_{j\neq i}^m \wh \bfU_
    {\tau_j}^{0,*}  \bfU_{\tau_j}^0 \bigr)^{\top} \bfU_{\tau_i}^{0,\top}  \right).
\end{align*}
Inserting the definition of $\bfQ_{\tau_i}^0$ and $\bfS_{\tau_i}^{0,\top}$ we obtain
\begin{align*}
    \wh\pi_{\tau,i}^\dagger \bigl( \wh X_\tau^0 \bigr) &= \ten_0 \left( \wh \bfQ_{\tau_i}^{0,*} \bfQ_{\tau_i}^0 \bfS_{\tau_i}^{0,\top} \bfU_{\tau_i}^{0,\top}  \right) \\
    &= \ten_0 \left( \wh \bfQ_{\tau_i}^{0,*} \bfQ_{\tau_i}^0 \bfS_{\tau_i}^{0,\top} \mat_0(Y_{\tau_i}^0)^\top  \right) = Y_{\tau_i}^0 \times_0 \wh \bfQ_{\tau_i}^{0,*} \bfQ_{\tau_i}^0 \bfS_{\tau_i}^{0,\top}.
\end{align*}
Hence, if $\tau_i = (\sigma_1,\dots,\sigma_k)$ is again a tree, we have
\begin{align*}
    \wh D_{\tau_i}^0
    = C_{\tau_i}^0 \times_0
      \wh\bfQ_{\tau_i}^{0,*} \bfQ_{\tau_i}^0 \bfS_{\tau_i}^{0,\top}
    \;\in\; \C^{2r_{\tau_i}^0 \times r_{\sigma_1}^0 \times \cdots \times r_{\sigma_k}^0} ,
\end{align*}
and if $\tau_i = \ell$ is a leaf, we obtain
\begin{align*}
    \wh\pi_{\tau,i}^\dagger \bigl( \wh X_\tau^0 \bigr) = \bfU_\ell^0 \bfS_\ell^0 \bfQ_\ell^{0,\top} \overline{\wh \bfQ_\ell^0}.
\end{align*}

In other words, the non-orthonormal factor
$\wh\bfQ_{\tau_i}^{0,*} \bfQ_{\tau_i}^0 \bfS_{\tau_i}^{0,\top}
\in \C^{2r_{\tau_i}^0 \times r_{\tau_i}^0}$
is multiplied into the $0$-mode of the $i$th subtree. Note that this moves the orthogonality center to the subtree $\tau_i$ and, at the same time, expresses the $0$-mode in the augmented basis $\wh\bfQ_{\tau_i}$. The recursion then proceeds with $\wh F_{\tau_i}$ and $\wh X_{\tau_i}^0$, completing the induction step. Finally, we obtain the initial data for the second-order parallel BUG integrator by projecting each $\wh X_{\tau_i}^0$ into the pre-augmented basis, i.e., we define 
\begin{align*}
    \wh Y_{\tau_i}^0 := \wh X_{\tau_i}^0 \times_0 \wh \bfU_{\tau_i}^0 \wh \bfU_{\tau_i}^{0,*}.
\end{align*}
% \jknote{$X$ is a tensor. Perhaps:
% \begin{align*}
%     \wh Y_\tau^0 := \ten_0(\wh \bfU_\tau^0 \wh \bfU_\tau^{0,*}\mat_0(\wh X_\tau^0)) 
% \end{align*}
% or
% \begin{align*}
%     \wh Y_\tau^0 := \wh X_\tau^0 \times_0 \wh \bfU_\tau^0 \wh \bfU_\tau^{0,*}\,.
% \end{align*}
% }
% \ds{You are right, adjusted! }

\subsubsection{Parallel time integration}\label{subsec:time_integration_2nd_parallel_TTN}
The crucial point of the second order parallel BUG integrator is that all basis matrices and core tensors can be evolved fully in parallel. With the definitions from Section~\ref{subsub:pre_augment_TTN} and \ref{subsub:constr_Q_TTN}, we distinguish the cases where a subtree is either a leaf or a core tensor:
\begin{enumerate}
    \item \textbf{Evolution of a leaf:}\\
Suppose $\tau_i = \ell$ is a leaf. Then we solve the $n_\ell \times 2r_\ell^0$ matrix differential equation from $t_0$ to $t_1 = t_0 + h$
    \begin{align*}
        \dot \bfY_\ell(t) = \wh \bfF_\ell(t,\bfY_\ell (t)), \quad \bfY_\ell(t_0) = \wh \bfU_\ell^0 \wh \bfU_\ell^{0,*} \mat_0(\wh X_\ell^0)^\top = \wh\bfU_\ell^0 \wh\bfU_\ell^{0,*} \bfU_\ell^0 \bfS_\ell^0 \bfQ_\ell^{0,\top} \overline{\wh \bfQ_\ell^0}.
    \end{align*}
    Construct $\wh \bfU_\ell^1 = \mathrm{orth} (\wh\bfU_\ell^0,Y_\ell (t_1)) = (\wh\bfU_\ell^0,\wt \bfU_\ell) \in \C^{n_\ell \times 4r_\ell^0}$ as an orthonormal basis of the range of the $n_\ell \times 4r_\ell^0$ matrix $(\wh \bfU_\ell^0,Y_\ell(t_1))$. Compute the tensor $\wt C_\ell^1 = \ten_i(\wt \bfU_\ell^* Y_\ell (t_1) \wh \bfQ_\ell^{0,\top}) \in \C^{2r_\tau^0 \times 2r_{\tau_1}^0\times \dots \times 2r_{\tau_m}^0}$.
    \item \textbf{Evolution of a core tensor:}\label{subsub:time_integration_TTN}\\
    Now suppose a tree $\tau = (\tau_1,\dots,\tau_m) \leq \bar \tau$. In the same spirit as for the leaves, we want to project the core tensor at time $t_0$ in the pre-augmented basis $\wh \bfU_\tau^0$. I.e., we want the initial data $\wh \bfU_\tau^0 \wh \bfU_\tau^{0,*} \mat_0(\wh X_\tau^0)^\top = \wh\bfU_\tau^0 \wh\bfU_\tau^{0,*} \bfU_\tau^0 \bfS_\tau^0 \bfQ_\tau^{0,\top} \overline{\wh \bfQ_\tau^0}$. This can be computed efficiently by
    \begin{align*}
        \wh\bfU_\tau^0 \wh\bfU_\tau^{0,*} \bfU_\tau^0 \bfS_\tau^0 \bfQ_\tau^{0,\top} \overline{\wh \bfQ_\tau^0} &= \mat_0(\wh Y_\tau^0)^\top \bfM_\tau \bfS_\tau^0 \bfQ_\tau^{0,\top} \overline{\wh \bfQ_\tau^0} \\
        &= \mat_0 \left( \wh Y_\tau^0 \times_0 \wh\bfQ_{\tau}^{0,*} \bfQ_{\tau}^0 \bfS_{\tau}^{0,\top} \bfM_\tau^\top \right)^\top. 
    \end{align*}
    where we used the identity $\bfM_\tau = \wh\bfU_\tau^{0,*} \bfU_\tau^0$. Hence, we multiply the core tensor of $\wh Y_\tau^0 = \ten_0 (\wh \bfU_\tau^{0,\top})$ by the non-orthogonal matrix $\wh\bfQ_{\tau}^{0,*} \bfQ_{\tau}^0 \bfS_{\tau}^{0,\top} \bfM_\tau^\top$. Thus, to evolve a core tensor in time, we integrate from $t_0$ to $t_1 = t_0 + h$ the $2r_\tau \times 2r_{\tau_1} \times \dots \times 2r_{\tau_m}$ tensor differential equation
    \begin{align}
        \dot{\bar C}_\tau (t) = \wh F_\tau(t,\bar C_\tau(t) \bigtimes_{i=1}^m \wh \bfU_{\tau_i}^0) \bigtimes_{i=1}^m \wh \bfU_{\tau_i}^{0,*}, \quad \bar C_\tau(t_0) = \wh C_\tau^0 \times_0 \wh\bfQ_{\tau}^{0,*} \bfQ_{\tau}^0 \bfS_{\tau}^{0,\top} \bfM_\tau^\top .
    \end{align}
\end{enumerate}

\subsubsection{Augmentation of core tensors}
The augmentation step is a sequential procedure from the leaves to the root. Note that on each level the augmentation strategy can be parallelized. On the lowest level of the leaves, we already have the new augmented basis $\wh \bfU_\ell^1 = [\wh \bfU_\ell^0,\wt \bfU_\ell]$ of size $n_\ell \times 4r_\ell^0$ and we computed the tensor
\begin{align*}
    \wt C_\ell^1 = \ten_i(\wt \bfU_\ell^* Y_\ell (t_1) \wh \bfQ_\ell^{0,\top}).
\end{align*}
Using these objects, we explain how to augment the core tensors recursively from the leaves to the root. 
Suppose a tree $\tau = (\tau_1,\dots,\tau_m) \leq \bar \tau$ where the subtrees $\tau_i$ are already augmented, for $i=1,\dots,m$. In analogy to the Tucker case, we augment $\bar C_\tau(t_1)$ in the modes $1,\dots,m$ to a tensor $\wh{\bar C}_\tau \in \C^{2r_\tau^0 \times 4r_{\tau_1}^0 \times \dots \times 4r_{\tau_m}^0}$ such that 
\begin{align*}
    &\wh{\bar C}_\tau \bigtimes_{i=1}^m \wh \bfU_{\tau_i}^{0,*} \wh \bfU_{\tau_i}^1 = \bar C_\tau (t_1), \\
    &\wh{\bar C}_\tau \bigtimes_{j \neq i}^m \wh \bfU_{\tau_j}^{0,*} \wh \bfU_{\tau_j}^1 \times_i \wt\bfU_{\tau_i}^{0,*} \wh \bfU_{\tau_i}^1  = \wt C_{\tau_i}^{1} \quad \text{for} \ i=1,\dots,m,
\end{align*}
and zero-valued entries elsewhere. By this choices, all modes but the $0$-mode have been augmented. If $\tau$ is the root, there is nothing more to augment and we set $\wh C_\tau^1 = \wh{\bar C}_\tau$. If $\tau < \bar \tau$, the $0$-mode augmentation is done by computing an orthonormal basis $\wh \bfQ_\tau^1$ of rank $\leq 4r_\tau^0$ of the matrix
\begin{align*}
    \wh \bfQ_\tau^1 &= \text{orth} \left[ \, \mat_0 \left(\wh C_\tau^0 \bigtimes_{i=1}^m \wh \bfU_{\tau_i}^{1,*} \wh \bfU_{\tau_i}^0 \right)^\top , \, \mat_0\left( \wh{\bar C}_\tau \right)^\top
       \right] \\ 
       &= \left[ \mat_0 \left(\wh C_\tau^0 \bigtimes_{i=1}^m \wh \bfU_{\tau_i}^{1,*} \wh \bfU_{\tau_i}^0 \right)^\top , \, \wt \bfW_{\tau} \right].
\end{align*}
Here, the first matricization, which is already orthogonal, denotes the pre-augmented core $\wh C_\tau^0$ projected in the augment basis $\wh \bfU_{\tau_i}^1$, $i=1,\dots,m$, while the second one is the approximation at time $t_1$. We then set 
\begin{align*}
    \wh C_\tau^1 := \ten_0 \left( \wh \bfQ_\tau^{1,\top}\right) \in \C^{4r_\tau^0 \times 4r_{\tau_1}^0 \times\dots \times 4r_{\tau_m}^0}.
\end{align*}
In total we have the fully augmented tree tensor network 
\begin{align*}
    \wh Y_\tau^1 = \wh C_\tau^1 \times_0 \wh \bfI_{\tau} \bigtimes_{i=1}^m \wh \bfU_{\tau_i}^1.
\end{align*}
For the recursion, it remains to define the tensors $\wt C_\tau^1$ for all $\tau \leq \bar \tau$. In the same spirit as for the leaves, we set
\begin{align*}
    \wt C_\tau^1 := \ten_i(\wt \bfU_\tau^* \bfK_\tau (t_1) \wh \bfQ_\tau^{0,\top}),
\end{align*}
where 
\begin{align*}
    \bfK_\tau(t_1) &= \mat_0(Z_\tau^K)^\top, \quad \text{with}  &&Z_\tau^K = \bar C_\tau (t_1) \times_0 \bfI_{\tau} \bigtimes_{i=1}^m \wh\bfU_{\tau_i}^0, \\
    \wt \bfU_\tau &= \mat_0(Z_\tau^U)^\top, \quad \text{with} &&Z_\tau^U = \ten_0(\wt \bfW_\tau^\top) \times_0 \bfI_{\tau} \bigtimes_{i=1}^m \wh \bfU_{\tau_i}^1.
\end{align*}
In total, we obtained a tree tensor network 
\begin{align*}
    \wh Y_{\bar \tau}^1 = \wh C_{\bar \tau}^1 \times_0 \bfI_{\bar \tau} \bigtimes_{i=1}^m \wh \bfU_{\bar{\tau}_i}^1,
\end{align*}
of tree ranks $(4r_\tau^0)_{\tau \leq \bar \tau}$.

\subsubsection{Truncation}
As in the Tucker case, we perform a standard truncation of the tree tensor network $\wh Y_{\bar \tau}^1$ with a user defined tolerance $\vartheta$. Since this is a standard procedure, we refer to \cite[Algorithm 7]{CLS23} for a detailed explanation. This results in a tree tensor network
\begin{align*}
    Y_{\bar \tau}^1 = C_{\bar \tau}^1 \times_0 \bfI_{\bar \tau} \bigtimes_{i=1}^m \bfU_{\bar{\tau}_i}^1,
\end{align*}
of tree ranks $(r_\tau^1)_{\tau \leq \bar \tau}$ with $r_\tau^1 \leq 4r_\tau^0$.

\subsection{Second-order augmented BUG for tree tensor networks}
The second-order augmented BUG integrator for Tucker tensors from Section~\ref{sec:2nd_order_augmented_BUG} can be extended to tree tensor networks in an analogous way, as done in \cite{CLS23}. We briefly describe the full integrator:

\begin{enumerate}
    \item \textbf{Pre-augment:} For all $\tau \leq \bar \tau$, construct the augmented basis $\wh \bfU_{\tau}^0$, the matrices $\bfQ_\tau^0$, $\bfS_\tau^{0}$, $\wh \bfQ_\tau^0$, $\wh \bfV_\tau^0$, $\bfV_\tau^0$ and the tensor $\wh C_{F_\tau}^0$ as described in the subsections before. 
    \item \textbf{Update of leaves:} Evolve and augment all leaves $\ell \in \mathcal{L}$ in the same way as for the second-order parallel BUG integrator, see Section~\ref{subsec:time_integration_2nd_parallel_TTN}.
    \item \textbf{Update of core tensors:} We update the core tensors recursively from the lowest level to the root. Let $\tau = (\tau_1,\dots,\tau_m)\leq \bar \tau$ and all subtrees $\tau_i$ are already updated and augmented, i.e., we have $\wh \bfU_{\tau_i}^1 \in \C^{n_{\tau_i} \times 4r_{\tau_i}^0}$. We then solve the $2r_\tau^0 \times 4r_{\tau_1}^0 \times \dots \times 4r_{\tau_m}^0$ tensor differential equation
    \begin{align*}
        \dot{\wh C_\tau}(t) &= \wh F_\tau( \wh C(t) \times_0 \bfI_{\tau} \bigtimes_{i=1}^m \wh \bfU_{\tau_i}^1) \bigtimes_{i=1}^m \wh \bfU_{\tau_i}^{1,*}, \\
        \wh C_\tau(t_0) &= C_\tau^0 \times_0 (
      \wh\bfQ_{\tau}^{0,*} \bfQ_{\tau}^0 \bfS_{\tau}^{0,\top}) \bigtimes_{i=1}^m \wh \bfU_{\tau_i}^{1,*} \bfU_{\tau_i}^0.
    \end{align*}
    Then compute an orthonormal basis $\wh \bfQ_\tau^1$ of rank $\leq 4r_\tau^0$ of the range of the matrix $[\mat_0 (\wh C_\tau(t_1))^\top, \mat_0 (\wh C_\tau(t_0))^\top]$. We set
    \begin{align*}
    \wh C_\tau^1 := \ten_0 \left( \wh \bfQ_\tau^{1,\top}\right) \in \C^{4r_\tau^0 \times 4r_{\tau_1}^0 \times\dots \times 4r_{\tau_m}^0}.
    \end{align*}
    \item \textbf{Truncation:} The resulted tree tensor network 
    \begin{align*}
        \wh Y_{\bar \tau}^1 = \wh C_{\bar \tau}^1 \times_0 \bfI_{\bar \tau} \bigtimes_{i=1}^m \wh \bfU_{\bar{\tau}_i}^1,
    \end{align*}
    of tree ranks $(4r_\tau^0)_{\tau \leq \bar \tau}$ needs to be truncated to tree ranks $(r_\tau^1)_{\tau \leq \bar \tau}$ with $r_\tau^1 \leq 4r_\tau^0$, see \cite[Algorithm 7]{CLS23}.
\end{enumerate}

\subsection{Robust error bound}
The robust error bound for the second-order parallel BUG for Tucker tensors and the robust error bound for the second-order augmented BUG for Tucker tensors extend to tree tensor networks in the same way as in \cite[Thm. 6.1]{CeLW20}. We therefore omit a detailed proof, but give the precise statement and its assumptions.

Due to the changing ranks, the tree tensor network manifold is different in every time step. Thus, for a tree $\tau=(\tau_1,\ldots,\tau_m)$ we use the notation $\M_\tau^k=\M_\tau((n_\tau)_{\tau\le\bar\tau},(r_\tau^k)_{\tau\le\bar\tau})$ for the TTN manifold in the $k$th time step. Further, we recall the corresponding tensor space $\mathcal{V}_\tau=\C^{r_\tau\times n_{\tau_1}\times \dots \times n_{\tau_m}}$. Further, we set $\mathcal{V} = \mathcal{V}_{\bar\tau}$ and $\mathcal{M}^k = \mathcal{M}_{\bar \tau}^k$ for the full tree. We make the following assumptions, cf. \cite{CLS23,CeLW20}:

\begin{enumerate}[label=A\arabic*, ref=A\arabic*]
	\item\label{ass:bounded_TTN} We assume $F:[0,t^*]\times \mathcal{V} \rightarrow \mathcal{V}$ to be bounded and Lipschitz continuous, i.e.,
    \begin{align*}
        \Vert F(Y) - F(\wt Y) \Vert &\leq L \Vert Y - \wt Y \Vert \quad &&\forall \ Y,\wt Y \in \mathcal{V}, \\
        \Vert F(Y) \Vert &\leq B \quad &&\forall \ Y \in \mathcal{V}, 
    \end{align*}
    This could be also weakened to a local Lipschitz condition and a local bound in the neighborhood of the exact solution $A(t)$.
	\item\label{ass:eps-normal_TTN} Given the projector $P$ onto the tangent space at $Y$, we assume that the function values $F(t,Y)$ are in the tangent space $\mathcal{T}_Y \M^k$ up to a small remainder: with $P^k_Y$ denoting the orthogonal projection onto $\mathcal{T}_{Y}\M^k$, we assume that for some $\eps>0$,
    \begin{align*}
        \| F(t,Y) - P^k_{Y} F(t,Y) \| &\le \eps
    \end{align*}
    for all $(t,Y) \in [t_k,t_{k+1}]\times \M^k$ in some ball $\| Y \| \le \rho$, where it is assumed that the exact solution $A(t)$, $0\le t \le t^*$, has a bound that is strictly smaller than $\rho$.
     
	\item\label{ass:init_TTN} The error at the initial condition is bounded by $\Vert Y^0 - A(0)\Vert \leq \delta$.
	\end{enumerate}

\begin{theorem}\label{thm:global_error_bound_TTN}
     Under the assumptions \ref{ass:bounded_TTN}-\ref{ass:init_TTN}, the error of the numerical approximation $Y_n$ at time $t_n = t_0 + nh$, obtained with $n$ time steps of the second-order parallel TTN integrator/second-order augmented TTN integrator with step size $h>0$ and rank-truncation tolerance $\vartheta$, is bounded by
     \begin{align}
         ||Y_n - A(t_n)|| \leq c_0 \delta + c_1 h^2 + c_2 h \varepsilon_r^{\mathrm{max}} + c_3 \varepsilon_{2r}^{\mathrm{max}} + c_4 n\vartheta,
     \end{align}
     where all appearing constants $c_0,\dots,c_4$ are independent of small singular values of matricizations of the connection tensors.
 \end{theorem}

We note that all favorable properties from Subsection~\ref{subsec:properties_2nd_augment} extend to the tree tensor network case. More precisely, the second-order augmented BUG integrator preserves norm and energy for Schrödinger equations and diminishes energy in gradient systems, up to the truncation tolerance. The proofs can follow again exactly the same arguments as in \cite[Sec. 6]{CLS23}.

\section{Numerical experiments}\label{sec:num_ex}
We verify the theoretical findings by applying both integrators to a problem from radiation transfer and a quantum spin system. All code is provided online \cite{code}.

\subsection{Uncertain radiative transfer}

To demonstrate the second-order convergence of the proposed integrator, we investigate the plane source benchmark with uncertain scattering coefficients. The plane source benchmark~\cite{ganapol2008analytical} is a standard test case in radiative transfer, which describes the time evolution of an isotropic particle beam in slab geometry. It is described by the radiative transfer equation
\begin{equation}
	\label{eq:rte}
	\begin{aligned}
		&\partial_t f + \mu \partial_x f + \sigma_s(\xi) f = \frac{\sigma_s(\xi) }{2} \int_{-1}^{1} f \,d\mu\,,
		\qquad 
		(x, \mu, \xi) \in [-2.5,2.5] \times [-1,1] \times \Xi, \\[1mm]
		& f(t = 0) = \frac{1}{\sqrt{2\pi}\cdot\sigma}\exp\Big(-\frac{x^2}{2\cdot \sigma^{2}} \Big)\,,
	\end{aligned}
\end{equation}
where the angular flux $f=f(t,x,\mu,\xi)$ describes the state of radiation particles. Here, $x\in [a,b]\subset\mathbb{R}$ is the spatial position, $\mu\in[-1,1]$ is the flight direction of radiation particles, and $\xi\sim U(\Xi)$ is a random variable uniformly distributed in $\Xi\subset\mathbb{R}$. The random variable models randomness in the isotropic scattering coefficient $\sigma_s$, that is not present in the original benchmark that commonly assumes $\sigma_s \equiv 1$. Here, we choose $\sigma_s(\xi) = 5 + 4\cdot \xi$ and $\Xi = [-1,1]$. Spatial derivatives are discretized using central differences combined with a Runge-Kutta 4 method in time. The random variable is discretized using a nodal method and we use a modal $P_N$ method to discretize the angular domain. We choose $100$ degrees of freedom in the space, time, and directional dimensions, leading to a total of $100^3$ unknowns. A full-rank reference solution is computed using $\Delta t = 10^{-3} \cdot \Delta x$, where $\Delta x = 0.05$ is the spatial cell length. We use time step sizes for the DLRA integrators of $h\in [0.1\cdot\Delta x, \Delta x]$. In Figure~\ref{fig:convergence_radiation} we can see the second-order convergence for the expected value, the variance and the full solution. As a comparison we further added the results of the first-order parallel BUG from \cite{ceruti2024parallel}.

\begin{figure}
    \centering
    \includegraphics[width=1.0\linewidth]{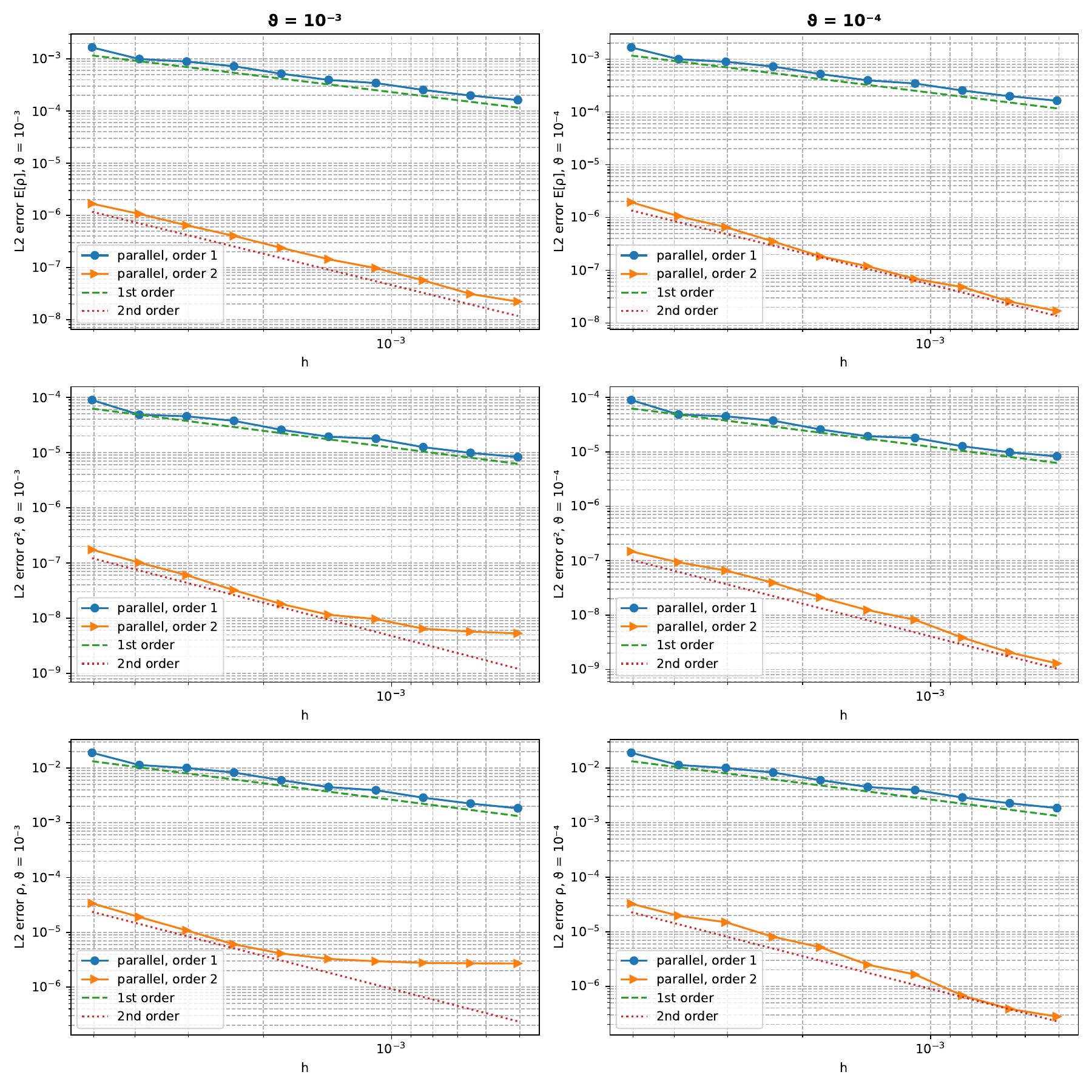}
    \caption{L$^2$ error of the numerical solution. Left column: Truncation tolerance of $\vartheta=10^{-3}$. Right column: Truncation tolerance of $\vartheta=10^{-4}$. Plotted are the errors of the expected value (top), the errors of the variance (center), and the error of the full solution (bottom).}
    \label{fig:convergence_radiation}
\end{figure}

\subsection{Schrödinger equation for quantum spin systems}\label{subsec:num_ex_QS_spin}
We illustrate the proposed integrators further for a problem in quantum physics. Consider the Ising model in a transverse field with next-neighbor interactions for $d$ spin-$\tfrac12$ particles, cf. \cite{stinchcombe1973ising}
\begin{align}
    i \partial_t \psi = H \psi, \quad \text{with} \quad H = -\Omega \sum_{k=1}^d \sigma_x^{(k)} - \sum_{k=1}^{d-1} \sigma_z^{(k)} \sigma_z^{(k+1)}. \label{eq:Schrödinger_spin}
\end{align}
Here $\sigma_x$ and $\sigma_z$ are the first and third Pauli matrices, respectively. Further, $\sigma^{(k)} = \bfI \otimes \dots \otimes \bfI \otimes \sigma \otimes \bfI \otimes \dots \otimes \bfI$ denotes the operator, where $\sigma$ acts only in the $kth$ site and $\Omega \geq 0$. For all the following simulations, we used the all spin up state as initial data, i.e., $\psi^0 = \bigotimes_{k=1}^d \begin{pmatrix}
    1 \\ 0
\end{pmatrix}$.

To verify the robust second-order error bound for the tree tensor network case, we apply both the parallel and the augmented second-order BUG integrator using a balanced binary tree tensor network. Each integrator is compared against the exact solution obtained from a direct computation of the matrix exponential $\exp(-itH)\psi^0$, which is still feasible for small system sizes. The displayed errors are measured in the Frobenius norm of the full tensor at the final time. The truncation tolerances are coupled to the step size as $ \vartheta_i = 0.1\,h_i^3$: since the rank-augmenting directions carry a weight of order $h^2$, the choice $\vartheta_i=\mathcal{O}(h_i^3)$  keeps the per-step truncation error below the local discretization error and thus preserves second-order. In Figure~\ref{fig:error_plot_spin}, we can observe the second-order convergence of the second-order parallel BUG and the second-order augmented BUG in the tree tensor network case. As expected, the error constant of the second-order augmented BUG integrator appears to be smaller compared to the second-order parallel integrator. Note that the second-order augmented BUG integrator shows higher order convergence for larger step-sizes, as due to the small system size, the order of the underlying classical Runge-Kutta method is more visible. For the smallest step-sizes both integrators saturate at the projection error $\eps$ as we set a maximal rank of $10$.

\begin{figure}
    \centering
    \includegraphics[width=1.0\linewidth]{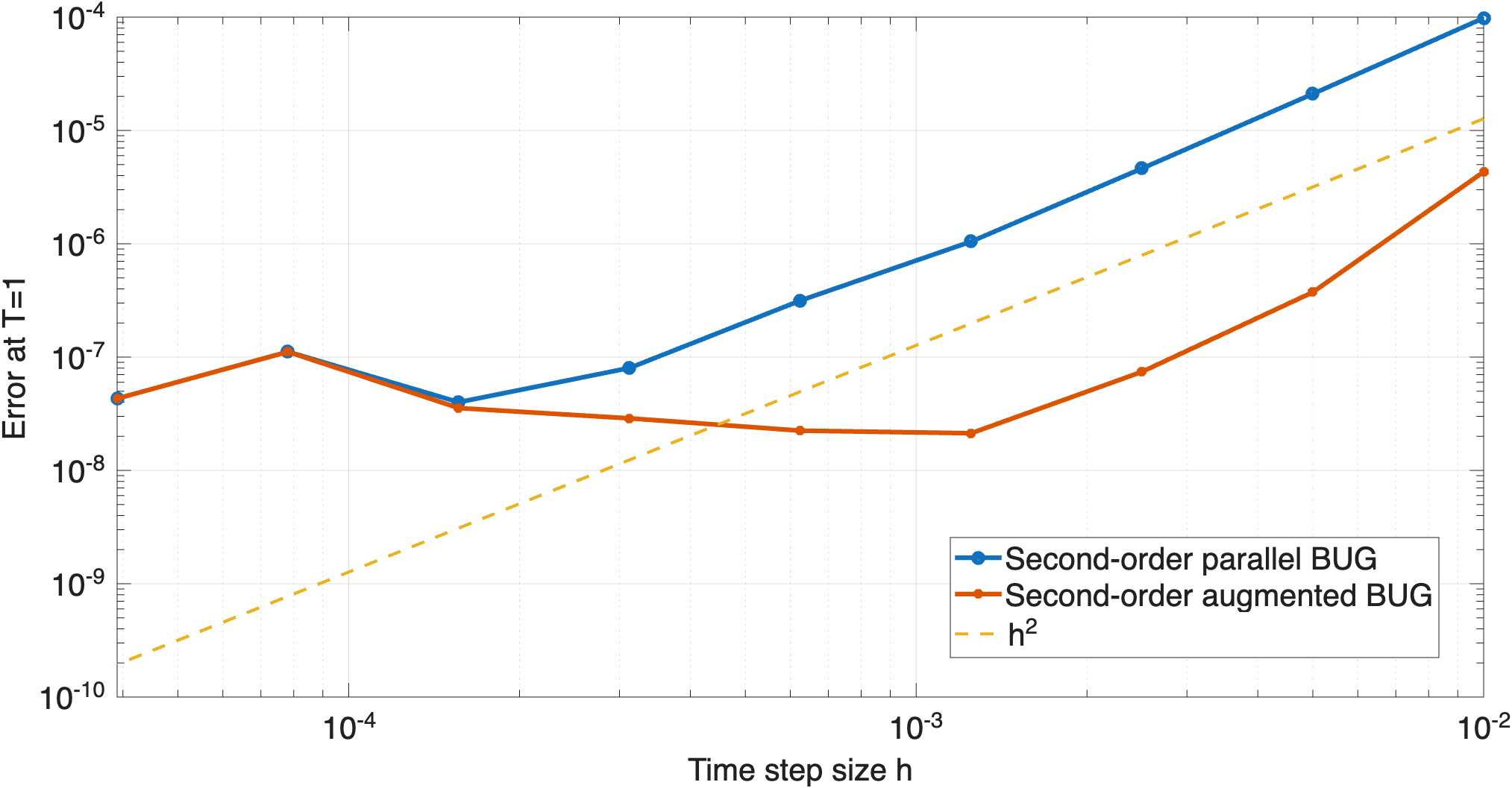}
    \caption{Errors at time $T=1$ of the second-order parallel BUG and second-order augmented BUG using a balanced binary tree tensor network. Unspecified parameters are $d=8$, $\Omega = 1$, and $r_{\text{max}} = 10$.}
    \label{fig:error_plot_spin}
\end{figure}

We further verify the structure-preserving properties established in Section~\ref{subsec:properties_2nd_augment}. Both integrators are applied to the Schrödinger equation~\eqref{eq:Schrödinger_spin}, and the norm and the energy are recorded at every time step. To isolate the conservation properties from the effect of rank truncation, the truncation tolerance is set to $\vartheta=10^{-14}$.  Figure~\ref{fig:norm_energy_spin_system} shows
that the augmented BUG integrator preserves norm and energy. The second-order parallel BUG integrator does not share this property by construction, since the augmented connecting tensor is assembled from approximated blocks. In this experiment, the resulting deviation is small; the norm is preserved up to $10^{-11}$ and the energy up to $10^{-7}$.

\begin{figure}
    \centering
    \includegraphics[width=1.0\linewidth]{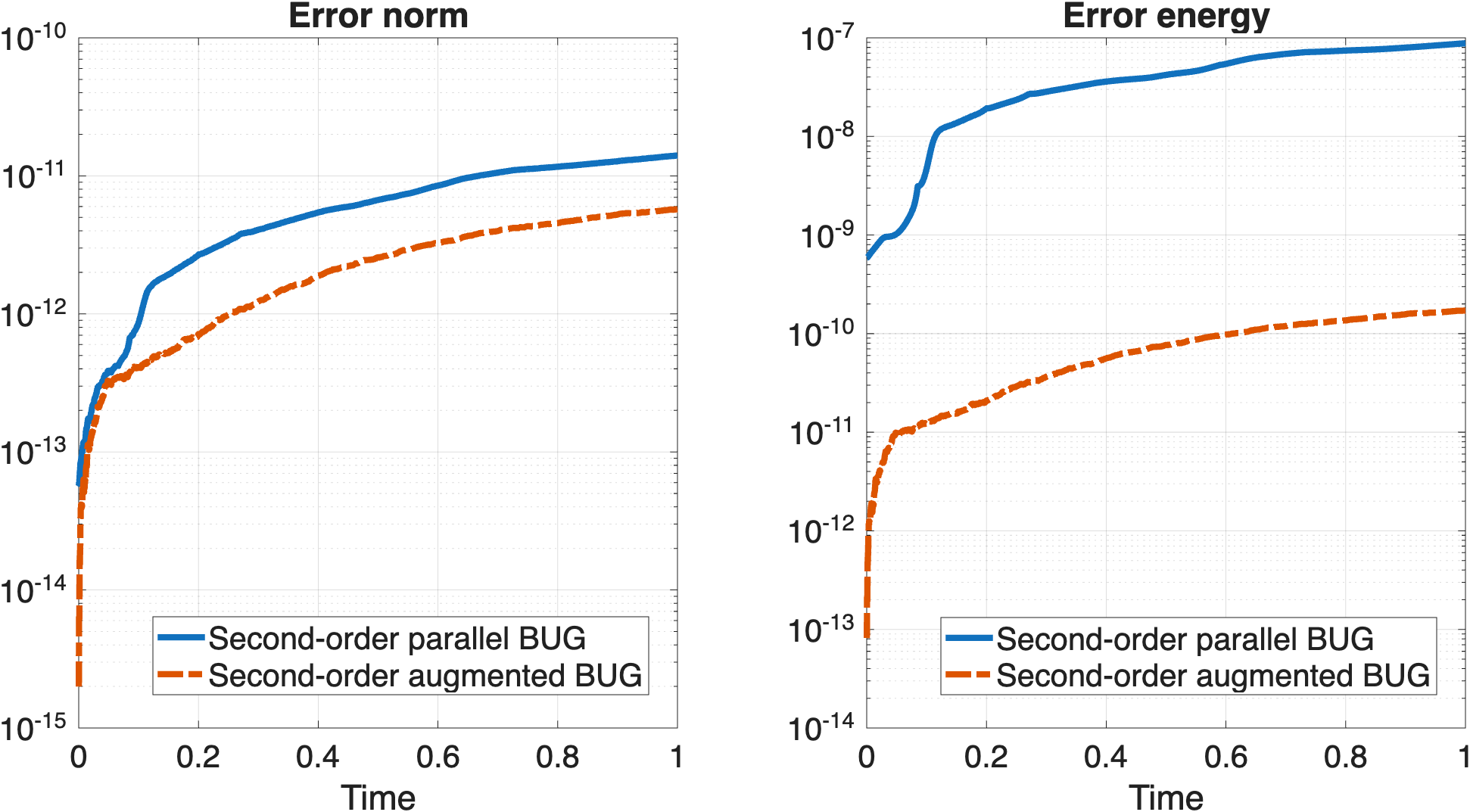}
    \caption{Norm and energy preservation for second-order parallel BUG and second-order augmented BUG. Left: Error of norm over time. Right: Error of energy over time. Unspecified parameters are $d=16$, $\Omega = 1$, $\vartheta = 10^{-14}$, $\tau = 0.00025$ and $r_{\text{max}} = 50$.}
    \label{fig:norm_energy_spin_system}
\end{figure}

% Possible examples for experiments:
% \begin{itemize}
%      \item Wärmeleitungsgleichung mit uncertainties (ähnlich wie midpoint BUG paper). Das könnte man mit Tucker machen. 3D Wärmeleitung + UQ
%      \item Ein TT/MPS Beispiel? 
%      \item Ein Quantum Beispiel mit Trees?
% \end{itemize} 

% \bibliographystyle{abbrv}
% \bibliographystyle{sn-mathphys-num}

\section*{Acknowledgments}
The work of Dominik Sulz was funded by the Deutsche Forschungsgemeinschaft (DFG, German Research Foundation) – TRR 352 – Project-ID 470903074.

\bibliography{main} 

% \section*{Appendix}

\end{document}